\pdfoutput=1
\documentclass{amsart}
\usepackage{amsmath,amssymb}
\usepackage{epsfig}
\usepackage{amsthm}

\newcommand{\N}{\mathbb{N}}

\newcommand{\R}{\mathbb{R}}
\newcommand{\C}{\mathbb{C}}

\newtheorem{theorem}{Theorem}[section]
\newtheorem{lemma}[theorem]{Lemma}
\newtheorem{corollary}[theorem]{Corollary}
\newtheorem{remark}[theorem]{Remark}

\def\x{\mathrm{x}}
\def\y{\mathrm{y}}
\def\z{\mathrm{z}}

\DeclareMathOperator*{\argmin}{arg\,min}

\usepackage[pdftex,colorlinks]{hyperref}
\hypersetup{linkcolor=black}
\hypersetup{urlcolor=black}
\hypersetup{citecolor=black}

\begin{document}
\title[An inverse boundary value problem for the $p$-Laplacian]{An inverse boundary value problem for the $p$-Laplacian}

\author{Antti Hannukainen}
\address{Aalto University, Department of Mathematics and Systems Analysis, P.O. Box 11100, FI-00076 Aalto, Finland} 
\email{antti.hannukainen@aalto.fi}

\author{Nuutti Hyv\"onen}
\address{Aalto University, Department of Mathematics and Systems Analysis, P.O. Box 11100, FI-00076 Aalto, Finland} 
\email{nuutti.hyvonen@aalto.fi}

\author{Lauri Mustonen}
\address{Emory University, Department of Mathematics and Computer Science, 400 Dowman Drive, Atlanta, GA 30322, USA} 
\email{lauri.mustonen@emory.edu}

\thanks{This work was supported by the Academy of Finland, the Aalto Science Institute and the Foundation for Aalto University Science and Technology.}



\begin{abstract}
This work tackles an inverse boundary value problem for a $p$-Laplace type partial differential equation parametrized by a smoothening parameter $\tau \geq 0$. The aim is to numerically test reconstructing a conductivity type coefficient in the equation when Dirichlet boundary values of certain solutions to the corresponding Neumann problem serve as data. The numerical studies are based on a straightforward linearization of the forward map, and they demonstrate that the accuracy of such an approach depends nontrivially on $1 < p < \infty$ and the chosen parametrization for the unknown coefficient. The numerical considerations are complemented by proving that the forward operator, which maps a H\"older continuous conductivity coefficient to the solution of the Neumann problem, is Fr\'echet differentiable, excluding the degenerate case $\tau=0$ that corresponds to the classical (weighted) $p$-Laplace equation. 
\end{abstract}

\maketitle

\section{Introduction}
\label{sec:intro}

This work considers an inverse boundary value problem for the $p$-Laplace type partial differential equation ($1 < p < \infty$)
\begin{equation}
\label{eq:intro}
\nabla \cdot \big( \sigma ( \tau^2 + |\nabla u|^2)^{\frac{p-2}{2}}\nabla u\big) \, = \, 0 \qquad {\rm in} \ \Omega,
\end{equation}
where $\Omega$ is a bounded Lipschitz domain and $\tau \geq 0$ is a smoothening parameter, with $\tau=0$ corresponding to the so-called weighted $p$-Laplacian \cite{Heinonen93}. To be more precise, the main aim is to numerically test reconstructing the strictly positive coefficient $\sigma \in L^\infty(\Omega)$ using Neumann--Dirichlet boundary value pairs of solutions to \eqref{eq:intro} as data. A partial differential equation of the type \eqref{eq:intro} can allegedly model several (physical) phenomena such as nonlinear dielectrics, plastic moulding, electro-rheological and thermo-rheological fluids, fluids governed by a power law, viscous flows in glaciology, or plasticity, but we emphasize that the main motivation for this manuscript is simply studying the properties of \eqref{eq:intro} as a nonlinear model (inverse) boundary value problem without any particular practical application in mind. However, spurred by the case $p=2$ and $\tau=0$ corresponding to the conductivity equation, we somewhat misleadingly refer to $\sigma$  as conductivity, to $u$ as potential and to its conormal derivative on $\partial \Omega$ as boundary current density. 

Although the case $p=2$ essentially corresponds to the inverse conductivity problem~\cite{Borcea02,Cheney99,Uhlmann09}, i.e.~the most studied inverse elliptic boundary value problem both theoretically and computationally, for $p\not=2$ the identifiability or reconstruction of $\sigma$ in \eqref{eq:intro}  from boundary data has not yet received much attention in the mathematical inverse problems literature. There essentially only exist results on the unique identifiability of the boundary trace $\sigma|_{\partial \Omega}$ \cite{Salo12} together with its first derivatives \cite{Brander16} and on the differentiation between two conductivities satisfying $\sigma_1 \leq \sigma_2$ \cite{Guo16}. In addition, the theoretical basis for the generalization of certain inclusion detection methods originally designed for the inverse conductivity problem, namely the monotonicity~\cite{Harrach13} and enclosure~\cite{Ikehata00} methods, has been laid in \cite{Brander17,Brander15}. On the other hand, we are not aware of any previous numerical studies on reconstructing $\sigma$ based on boundary values of solutions to \eqref{eq:intro}, though numerically implementing the enclosure method for \eqref{eq:intro} seems viable; see~\cite{Brander17,Brander15}. In particular, no one has previously considered the straightforward approach of linearizing the dependence of the solutions to \eqref{eq:intro} on $\sigma$ and numerically solving the ensuing linear inverse problem. Take note that such a linearization comprises the basic building block for iterative Newton-type algorithms, which are the most commonly used reconstruction methods in practical applications associated to inverse elliptic boundary value problems; cf.,~e.g.,~\cite{Arridge99, Cheney99, Soleimani05}.

In this work, we consider the Neumann boundary value problem for \eqref{eq:intro} with certain boundary current densities and treat the Dirichlet traces of the associated potentials as the data for the inverse problem. As the theoretical foundation for our linearization approach, we prove that the solution  to the Neumann problem for \eqref{eq:intro}, say $u_\sigma$, is Fr\'echet differentiable with respect to a H\"older continuous $\sigma$ if $\tau > 0$ or $u_\sigma$ has no critical points in $\overline{\Omega}$. The associated Fr\'echet derivative is defined by a solution to a certain anisotropic (linear) conductivity equation with a homogeneous Neumann condition.

In our two-dimensional numerical studies, the accuracy of the linearization approach is tested on the forward map taking $\sigma$ to $u_\sigma|_{\partial \Omega}$ as well as in the computation of an (approximate) {\em maximum a posteriori} (MAP) estimate for the conductivity from boundary data. It turns out that the corresponding errors depend nontrivially on $p$ and the chosen parametrization for $\sigma$, i.e., whether the linearization is computed with respect to the conductivity itself $\sigma$, the resistivity $1/\sigma$, some other power of $\sigma$ or the log-conductivity $\log \sigma$; see \cite{Hyvonen18} for similar considerations when $p=2$. In particular, one definitely cannot draw the tempting conclusion that the inverse boundary value problem for \eqref{eq:intro} is least nonlinear in the case $p=2$ that corresponds to a linear partial differential equation. Moreover, although our proof of Fr\'echet differentiability does not cover the classical $p$-Laplace equation, i.e.~$\tau=0$, the conclusions of our numerical experiments do not seem to depend much on the choice of a small or {\em vanishing} $\tau \geq 0$ in \eqref{eq:intro}.

This text is organized as follows. Section~\ref{sec:setting} introduces the mathematical framework and considers H\"older continuity of solutions to \eqref{eq:intro} with respect to $\sigma$. The Fr\'echet differentiability result is proved in Section~\ref{sec:Frechet} and the numerical experiments are documented in Section~\ref{sec:numerics}. Finally, Section~\ref{sec:conclusion} presents the concluding remarks. Some useful inequalities are collected in Appendix~\ref{app:A}.

\section{The setting and preliminary continuity results}
\label{sec:setting}
In what follows, we will constantly employ the quotient Sobolev spaces $W^{1,p}(\Omega)/\R$ equipped with the norm
$$
\| \nabla v \|_{L^p(\Omega)} \leq \| v \|_{W^{1,p}(\Omega)/\R} := \inf_{c\in \R} \| v - c \|_{W^{1,p}(\Omega)} \leq C(p,\Omega) \| \nabla v \|_{L^p(\Omega)},
$$
where the last inequality is a straightforward consequence of the Poincar\'e inequality. Here and in what follows, $\Omega \subset \R^n$, $n \in \N \setminus \{ 1\}$, is a bounded Lipschitz domain. Moreover, $C$ and $c$ denote positive constants that may change between different occurrences. The multiplier field of all considered function spaces is $\R$.

We define a family of `smoothened $p$-energy functions' via
\begin{equation}
\label{eq:loc_energy}
\varphi_{p,\tau}(\x) = \frac{1}{p}\big(\tau^2 + |\x|^2\big)^{\frac{p}{2}}, \qquad \x \in \R^n, \ 1 < p < \infty, \ \tau \geq 0,
\end{equation}
and note that the corresponding gradient is
\begin{equation}
\label{eq:varphi}
D \varphi_{p,\tau}(\x) = (\tau^2 + |\x|^2)^{\frac{p-2}{2}} \x. 
\end{equation}
We use roman $\x$ in the argument of $\varphi_{p,\tau}$ to avoid confusion with the actual spatial variable $x \in \R^n$. Moreover, the gradient of $\varphi_{p,\tau}$ is denoted by $D\varphi_{p,\tau}$ in order to reserve the standard $\nabla$-notation for the spatial derivatives appearing in the considered partial differential equations. 
Appendix~\ref{app:A} provides more information on fundamental properties of $\varphi_{p,\tau}$. 

\subsection{Neumann problem and its stable solvability}
Let $\sigma \in L^\infty_+(\Omega)$ be a `conductivity' living in
$$
L^\infty_+(\Omega) = \{ \upsilon \in L^\infty(\Omega) \, : \, {\rm ess} \inf \upsilon > 0 \}.
$$
We consider a ($\sigma$-weighted) $p$-Laplace type equation with a Neumann boundary condition:
\begin{equation}
\label{eq:plaplace}
\left\{
\begin{array}{ll}
\nabla \cdot \big( \sigma  D\varphi_{p,\tau} (\nabla u) \big) = 0 \quad & {\rm in} \ \Omega,  \\[2mm]
\nu \cdot \sigma D\varphi_{p,\tau} (\nabla u) = f \quad & {\rm on} \ \partial \Omega,
\end{array}
\right.
\end{equation}
where $1 < p < \infty$ and $\nu \in L^\infty(\partial \Omega, \R^{n-1})$ is the exterior unit normal of $\partial \Omega$. In particular, the first line of \eqref{eq:plaplace} reduces to the standard $p$-Laplace equation if $\tau=0$ and $\sigma \equiv 1$. The weak formulation of \eqref{eq:plaplace} is to find $u_\sigma \in W^{1,p}(\Omega) / \R$ such that
\begin{equation}
\label{eq:varplaplace}
\int_\Omega \sigma D \varphi_{p,\tau}(\nabla u_\sigma) \cdot \nabla v \, {\rm d} x 
= \int_{\partial \Omega} f v \, {\rm d} S \qquad {\rm for} \ {\rm all} \ v \in W^{1,p}(\Omega)/\R.
\end{equation}
It is well known that \eqref{eq:varplaplace} has a unique solution for any `boundary current density' $f \in L^q_{\diamond}(\partial \Omega)$ and $\tau \geq 0$, with $L^q_{\diamond}(\partial \Omega)$ standing for the zero-mean subspace of $L^q(\partial \Omega)$ and $q := p/(p-1)$ being the conjugate index of $p$. Moreover, this solution uniquely minimizes the (weighted and smoothened) $p$-energy
\begin{equation}
\label{eq:energy}
\mathcal{E}(v) := \int_{\Omega} \sigma \varphi_{p,\tau}(\nabla v) \, {\rm d} x - \int_{\partial \Omega} f v \, {\rm d} S
\end{equation}
over $v \in W^{1,p}(\Omega) / \R$.
However, as the Neumann boundary condition is rarely considered in the literature on the $p$-Laplace equation, we summarize these results as a theorem accompanied by a sketch of a proof. In what follows, we denote by
$$
f_\sigma = \frac{1}{{\rm ess} \inf \sigma} \, f  \in L^q_\diamond(\partial \Omega)
$$
a scaled version of the boundary current density in \eqref{eq:plaplace}.

\begin{theorem}
\label{thm:alku}
The problem \eqref{eq:varplaplace} has a unique solution that satisfies
\begin{equation}
\label{eq:upperbound}
\|  u_\sigma \|_{W^{1,p}(\Omega)/\R} \leq 
C \left\{
\begin{array}{ll}
\| f_\sigma \|_{L^q(\partial \Omega)}^{\frac{q}{p}} + \tau^{\frac{q-p}{p/q+q/p}} \| f_\sigma \|_{L^q(\partial \Omega)}^{\frac{2}{p/q +q/p}} & \quad 1 < p \leq 2, \\[2mm]
\| f_\sigma \|_{L^q(\partial \Omega)}^{\frac{q}{p}} & \quad 2 \leq p < \infty,
\end{array}
\right.
\end{equation}
where $C = C(\Omega,p) > 0$ does not depend on $f \in L^q_\diamond(\partial \Omega)$, $\sigma \in L^\infty_+(\Omega)$ or $\tau \geq 0$. In addition,
\begin{equation}
\label{eq:nabq}
\| D \varphi_{p,\tau}(\nabla u_\sigma)\|_{L^q(\Omega)} \leq 
C \left\{
\begin{array}{ll}
\!\! \|f_\sigma\|_{L^q(\partial \Omega)} + \tau^{\frac{2-p}{p/q + q/p}}\| f_\sigma \|_{L^q(\partial \Omega)}^{\frac{p}{p/q +q/p}} & \quad {\rm if} \ 1 < p \leq 2, \\[2mm]
 \!\! \|f_\sigma\|_{L^q(\partial \Omega)}  + \tau^{p-2} \|f_\sigma \|_{L^q(\partial \Omega)}^{\frac{q}{p}}   & \quad {\rm if} \ 2 \leq p < \infty,
\end{array}
\right.
\end{equation}
where again $C = C(\Omega,p) > 0$ does not depend on $f \in L^q_\diamond(\partial \Omega)$, $\sigma \in L^\infty_+(\Omega)$ or $\tau \geq 0$.
\end{theorem}

\begin{proof}
First of all, it is unambiguous to define \eqref{eq:varplaplace} and \eqref{eq:energy} on a quotient space since the gradient does not see an additive constant and $f$ has zero mean. Moreover, all integrals in \eqref{eq:varplaplace} and \eqref{eq:energy} are finite for all elements of the associated function spaces due to H\"older's inequality, \eqref{eq:welldef} and the trace theorem.

Since $\varphi_{p,\tau}:\R^n \to \R$ is strictly convex (see Appendix~\ref{app:A}), the fact that \eqref{eq:energy} has a unique minimizer in $W^{1,p}(\Omega) / \R$ and that this minimizer also uniquely solves \eqref{eq:varplaplace} follows by the same logic as in the case of a Dirichlet boundary condition; see,~e.g.,~\cite[Proposition~A.1]{Salo12} for the case $\tau = 0$. Indeed, compared to the proof of \cite[Proposition~A.1]{Salo12}, one essentially only needs to use \eqref{eq:convex} whenever \cite{Salo12} resorts to \cite[(A.5)]{Salo12} and to note that the (bounded) trace map retains weak convergence of a minimizing sequence. See also \cite{Maly16}.

In the rest of this proof the generic constant $C>0$ depends only on $\Omega$ and $p$.
Let us consider $2\leq p < \infty$. To deduce \eqref{eq:upperbound}, choose $v=u_\sigma$ in \eqref{eq:varplaplace} and apply H\"older's inequality, which leads to
\begin{align}
\label{eq:upper}
{\rm ess} \inf \sigma \, \| u_\sigma \|_{W^{1,p}(\Omega)/\R}^p 
&\leq C \int_{\Omega} \sigma D\varphi_{p,\tau}(\nabla u_\sigma) \cdot \nabla u_\sigma  \, {\rm d} x   = C \int_{\partial \Omega} f u_\sigma  \, {\rm d} S  \nonumber \\[1mm]
&\leq C \| f \|_{L^q(\partial \Omega)} \| u_\sigma \|_{L^p(\partial \Omega)/\R} \leq  C \| f \|_{L^q(\partial \Omega)} \| u_\sigma \|_{W^{1,p}(\Omega)/\R},
\end{align}
where the last step is an easy consequence of the trace theorem (cf.~\cite[Lemma~2.7]{Hyvonen04}). Dividing by ${\rm ess} \inf \sigma \, \| u_\sigma \|_{W^{1,p}(\Omega)/\R}$ and taking the $(p-1)$th root proves \eqref{eq:upperbound}. On the other hand, a direct estimation based on \eqref{eq:isop}, the triangle inequality and the continuity of the embedding $L^q(\Omega) \hookrightarrow L^p(\Omega)$ yields
\begin{align*}
\| D\varphi_{p,\tau}(\nabla u_\sigma) \|_{L^q(\Omega)} &\leq C \Big(
\tau^{p-2} \| \nabla u_\sigma \|_{L^q(\Omega)} + \big\| \, |\nabla u_\sigma|^{p-1} \big\|_{L^q(\Omega)} \Big) \\[1mm]
&\leq C \Big(
\tau^{p-2} \| \nabla u_\sigma \|_{L^p(\Omega)} + \big\| \nabla u_\sigma \big\|_{L^p(\Omega)}^{p-1} \Big).
\end{align*}
Substituting \eqref{eq:upperbound} in this estimate completes the proof for $2 \leq p < \infty$.

It remains to prove \eqref{eq:upperbound} and \eqref{eq:nabq} for $1 < p < 2$. First of all, notice that the estimate \eqref{eq:upper} still holds apart from its first inequality.  In particular, \eqref{eq:pienip} indicates that
\begin{align}
\label{eq:pp2lm}
{\rm ess} \inf \sigma \, \| D\varphi_{p,\tau}(\nabla u_\sigma) \|_{L^q(\Omega)}^{q} &\leq   \int_{\Omega} \sigma D\varphi_{p,\tau}(\nabla u_\sigma) \cdot \nabla u_\sigma \, {\rm d} x \nonumber \\[1mm]
&\leq  C \| f \|_{L^q(\partial \Omega)} \| u_\sigma \|_{W^{1,p}(\Omega)/\R}.
\end{align}
On the other hand, by virtue of \eqref{eq:pienip2},
\begin{align}
\label{eq:pp2p}
\| u_\sigma \|_{W^{1,p}(\Omega)/\R}^p &\leq C \Big( \int_{\Omega} D\varphi_{p,\tau}(\nabla u_\sigma) \cdot \nabla u_\sigma \, {\rm d} x \nonumber \\ 
&\qquad \qquad \qquad + \tau^{2-p}
\int_{\Omega} (\tau^2 + |\nabla u_{\sigma} |^2)^{\frac{p-2}{2}} |\nabla u_{\sigma}|^p \, {\rm d} x \Big).
\end{align} 
The second term on the right-hand side of \eqref{eq:pp2p} can be estimated as follows:
\begin{align*}
\int_{\Omega} &\frac{|\nabla u_{\sigma}|^p}{(\tau^2 + |\nabla u_{\sigma} |^2)^{\frac{2-p}{2}}} \, {\rm d} x = \int_{\Omega} \frac{|\nabla u_{\sigma} |^{(2-p) + (2p-2)}}{(\tau^2 + |\nabla u_{\sigma} |^2)^{\frac{(2-p)((2-p) + (p-1))}{2}}}  \, {\rm d} x \\[2mm]
& \qquad \leq \bigg( \int_{\Omega} \frac{|\nabla u_{\sigma} |}{(\tau^2 + |\nabla u_{\sigma} |^2)^{\frac{2-p}{2}}} \, {\rm d} x \bigg)^{2-p} \bigg(\int_{\Omega} \frac{|\nabla u_{\sigma} |^2}{(\tau^2 + |\nabla u_{\sigma} |^2)^{\frac{2-p}{2}}} \, {\rm d} x\bigg)^{p-1} \\[1mm]
& \qquad \leq C \| D\varphi_{p,\tau}(\nabla u_\sigma) \|_{L^1(\Omega)}^{2-p} \Big(\int_{\Omega} D\varphi_{p,\tau}(\nabla u_\sigma) \cdot \nabla u_\sigma \, {\rm d} x \Big)^{p-1},
\end{align*}
where the penultimate step is H\"older's inequality with the conjugate exponents $r = 1/(2-p)$ and $r' = 1/(p-1)$. With the help of \eqref{eq:pp2lm} and the continuous embedding $L^1(\Omega) \hookrightarrow L^q(\Omega)$, the estimate \eqref{eq:pp2p} thus leads to 
\begin{align}
\label{eq:pp2ml2}
\| u_\sigma \|_{W^{1,p}(\Omega)/\R}^p &\leq 
C \Big(\| f_\sigma \|_{L^q(\partial \Omega)} \|u_\sigma\|_{W^{1,p}(\Omega)/\R} \nonumber \\[0mm]
& \qquad \quad + \tau^{2-p} \big( \| f_\sigma \|_{L^q(\partial \Omega)} \|u_\sigma\|_{W^{1,p}(\Omega)/\R}\big)^{\frac{2(p-1)}{p}} \Big).
\end{align}
As $\| u_\sigma \|_{W^{1,p}(\Omega)/\R}^p$ must be smaller than two times the larger term on the right-hand side of \eqref{eq:pp2ml2}, this proves \eqref{eq:upperbound} for $1 < p \leq 2$ after dividing by the appropriate power of $\| u_\sigma \|_{W^{1,p}(\Omega)/\R}$ and algebraically manipulating the exponents. Finally, plugging \eqref{eq:upperbound} in \eqref{eq:pp2lm} straightforwardly validates \eqref{eq:nabq} and completes the proof.
\end{proof}
 
Notice that for $\tau = 0$ or $p=2$, the two upper bounds both in \eqref{eq:upperbound} and  in \eqref{eq:nabq} coincide, which is in line with the theory for the standard $p$-Laplace equation and for linear elliptic equations, respectively. In addition, it is easy to check that on the right-hand sides of both \eqref{eq:upperbound} and \eqref{eq:nabq} the term depending on $\tau$ dominates the other summand for any fixed $\tau > 0$ when $\|f\|_{L^q(\partial \Omega)} \to 0$ and their roles are reversed when $\|f\|_{L^q(\partial \Omega)} \to \infty$.

\subsection{Complementary perturbation estimate}
One can straightforwardly deduce Lipschitz and H\"older continuity of the forward map $L^\infty_+(\Omega) \ni \sigma \mapsto u_\sigma \in W^{1,p}(\Omega)/\R$ for $1 < p \leq 2$ and $2 < p < \infty$, respectively. Although this assertion is proved for a more general partial differential equation, $\tau=0$ and a Dirichlet boundary condition in \cite[Lemma~3.2]{Guo16}, we anyway formulate it as a lemma and present a brief proof for the sake of completeness, including a rather explicit dependence on $f$ and $\tau$ in the process. In the following, we denote the functions defined by the right hand sides of \eqref{eq:upperbound} and \eqref{eq:nabq} by $d(f_\sigma, p, \tau)$ and $\tilde{d}(f_\sigma, p, \tau)$, respectively, that is, $d(f_\sigma, p, \tau)$ provides an upper bound for $\| u_\sigma \|_{W^{1,p}(\Omega) / \R}$ and $\tilde{d}(f_\sigma, p, \tau)$ for $\| D \varphi_{p, \tau}(\nabla u_\sigma) \|_{L^q(\Omega)}$.

\begin{lemma}
\label{lemma:perus}
Let $u_{\sigma_0}, u_{\sigma_1} \in W^{1,p}(\Omega)/\R$ be the solutions of \eqref{eq:varplaplace} corresponding to $\sigma_0, \sigma_1 \in L^\infty_+(\Omega)$, respectively. Then, it holds that
\begin{equation}
\label{eq:conv}
\| \nabla u_{\sigma_1} - \nabla u_{\sigma_0} \|_{L^p(\Omega)} \leq C' \| \sigma_1 - \sigma_0 \|_{L^\infty(\Omega)}^{\min\{1,\frac{q}{p}\}}.
\end{equation}
The constant $C' >0$ admits a representation
\begin{equation}
\label{eq:Cprime}
C' = C \left\{ 
\begin{array}{ll}
 \underline{\sigma}^{-1} \Big( \tau^{2-p} + d\big(f_{\underline{\sigma}}, \tau, p\big)^{2-p}\Big) \, \tilde{d}\big(f_{\underline{\sigma}}, \tau, p\big) & \quad {\rm if} \ 1 < p \leq 2, \\[3mm]
\underline{\sigma}^{-\frac{q}{p}} \, \tilde{d} \big(f_{\underline{\sigma}}, \tau, p\big)^{\frac{q}{p}} & \quad {\rm if} \ 2 \leq p < \infty,
\end{array}
\right.
\end{equation}
where $\underline{\sigma} = \min_{j=0,1} \{{\rm ess} \inf \sigma_j \}$ and $C = C(\Omega, p)>0$ is independent of $\sigma_0$, $\sigma_1$, $f$ and $\tau\geq 0$. 
\end{lemma}

\begin{proof}
In this proof the generic constant $C>0$ depends only on $\Omega$ and $p$.

Following the main line of reasoning in the proof of~\cite[Lemma~3.2]{Guo16}, we define
\begin{align*}
I &:= \int_\Omega \sigma_0\big( \tau^2 + |\nabla u_{\sigma_1}|^2 + |\nabla u_{\sigma_0}|^2\big)^{\frac{p-2}{2}} | \nabla u_{\sigma_1} - \nabla u_{\sigma_0} |^2 \, {\rm d} x \\[1mm]
&\ \leq  C \int_{\Omega} \sigma_0\big( D\varphi_{p,\tau}(\nabla u_{\sigma_1}) -  D\varphi_{p,\tau}(\nabla u_{\sigma_0}) \big) \cdot (\nabla u_{\sigma_1}- \nabla u_{\sigma_0}) \, {\rm d} x,
\end{align*}
where the inequality holds because of \eqref{eq:basic_ineq}. By subtracting the variational equations~\eqref{eq:varplaplace} for $\sigma_0$ and $\sigma_1$ with $v=u_{\sigma_1} - u_{\sigma_0}$, it thus follows that 
\begin{align}
\label{eq:I_ylaraja}
I &\leq C \int_{\Omega} (\sigma_0 - \sigma_1) D\varphi_{p,\tau}(\nabla u_{\sigma_1}) \cdot (\nabla u_{\sigma_1}- \nabla u_{\sigma_0}) \, {\rm d} x \nonumber \\[2mm]
& \leq C \| \sigma_1 - \sigma_0 \|_{L^\infty(\Omega)} \| D\varphi_{p,\tau}(\nabla u_{\sigma_1}) \|_{L^q(\Omega)} \| \nabla u_{\sigma_1}- \nabla u_{\sigma_0}\|_{L^p(\Omega)},
\end{align}
where we also used H\"older's inequality.

If $2 \leq p < \infty$, then obviously
$$
\| \nabla u_{\sigma_1} - \nabla u_{\sigma_0} \|^p_{L^p(\Omega)} \leq  C \frac{I}{{\rm ess} \inf \sigma_0}.
$$
Hence, the claim for $2 \leq p < \infty$ follows by using \eqref{eq:I_ylaraja} and \eqref{eq:nabq}, dividing by $\| \nabla u_{\sigma_1}- \nabla u_{\sigma_0}\|_{L^p(\Omega)}$ and taking the $(p-1)$th root.

To complete the proof, let $1< p \leq 2$. Writing
$$
1 = \big(\tau^2+|\nabla u_{\sigma_1}|^2 + |\nabla u_{\sigma_0}|^2\big)^{\frac{p(2-p)}{4}} \big(\tau^2+|\nabla u_{\sigma_1}|^2 + |\nabla u_{\sigma_0}|^2\big)^{\frac{p(p-2)}{4}}
$$
and applying H\"older's inequality with the conjugate exponents $r = 2/(2-p)$ and $r'=2/p$, one can deduce that
\begin{align*}
\| \nabla u_{\sigma_1} - \nabla u_{\sigma_0} \|^p_{L^p(\Omega)} &\leq \Big(\int_\Omega  \big(\tau^2+|\nabla u_{\sigma_1}|^2 + |\nabla u_{\sigma_0}|^2\big)^{\frac{p}{2}}  {\rm d} x\Big)^{\frac{2-p}{2}} \times \\[0mm]
& \qquad \times \Big( \int_\Omega  \big(\tau^2+|\nabla u_{\sigma_1}|^2 + |\nabla u_{\sigma_0}|^2\big)^{\frac{p-2}{2}}  |\nabla u_{\sigma_1} - \nabla u_{\sigma_0} |^2 \, {\rm d} x \Big)^{\frac{p}{2}}  \\[0mm]
& \leq C \big(\tau^p+\|\nabla u_{\sigma_1}\|_{L^p(\Omega)}^p + \|\nabla u_{\sigma_0}\|_{L^p(\Omega)}^p\big)^{\frac{2-p}{2}} \, \left(\frac{I}{{\rm ess} \inf \sigma_0}\right)^{\frac{p}{2}} \\[0mm]
& \leq C \Big(\tau^{\frac{p(2-p)}{2}} + d(f_{\underline{\sigma}},p,\tau)^{\frac{p(2-p)}{2}} \Big) \, \left(\frac{I}{{\rm ess} \inf \sigma_0}\right)^{\frac{p}{2}},
\end{align*}
where the last step follows from \eqref{eq:upperbound}.
After employing \eqref{eq:I_ylaraja} and \eqref{eq:nabq}, the claim for $1< p \leq 2$ follows by dividing with $\| \nabla u_{\sigma_1}- \nabla u_{\sigma_0}\|_{L^p(\Omega)}^{p/2}$ and taking the $(2/p)$th power of the resulting inequality.
\end{proof}

It is worth noting that \eqref{eq:Cprime} takes the form
$$
C' = C \left\{ 
\begin{array}{ll}
 \underline{\sigma}^{-q} \| f \|_{L^q(\partial \Omega)}^{\frac{q}{p}}  & \quad {\rm if} \ 1 < p \leq 2, \\[2mm]
\underline{\sigma}^{-\frac{2q}{p}} \| f \|_{L^q(\partial \Omega)}^{\frac{q}{p}} & \quad {\rm if} \ 2 \leq p < \infty,
\end{array}
\right.
$$
when $\tau = 0$.

For our purposes it is important to attain Lipschitz continuity of the forward operator for all $1 < p < \infty$. This  will be achieved by assuming more regularity from the problem setting in the following subsection.

\subsection{H\"older conductivities}
Suppose $\partial \Omega$ is of the H\"older class $\mathcal{C}^{1,\alpha}$, conductivities live in $\mathcal{C}^{\alpha}(\overline{\Omega}) \cap L^\infty_+(\Omega)$ and the boundary current density in \eqref{eq:plaplace} satisfies $f\in \mathcal{C}^{\alpha}(\partial \Omega) \cap L^q_\diamond(\partial \Omega)$ for some $\alpha > 0$. Under these assumptions,
\begin{equation}
\label{eq:uniform}
\| \nabla u_{\sigma} \|_{\mathcal{C}^{\beta}(\overline{\Omega})} \leq C_{\mathcal{B}}
\end{equation}
for all $\sigma$ in any bounded subset $\mathcal{B}$ of $\mathcal{C}^{\alpha}(\overline{\Omega}) \cap L^\infty_+(\Omega)$ for which
$$
\inf_{\sigma \in \mathcal{B}}  \big({\rm ess} \inf \sigma \big) > 0.
$$
The constants $\beta > 0$ and $C_{\mathcal{B}} > 0$ in \eqref{eq:uniform} depend, in addition to $\mathcal{B}$,  on $\Omega$, $\alpha$, $1 < p < \infty$, $f$ and $\tau \geq 0$. See \cite[Theorem~2]{Liebermann88} for the details; cf.~ also \eqref{eq:posdef}, \eqref{eq:bound} and \cite[Section~5]{Maly16}. In what follows, $\mathcal{B}$ always denotes a subset of $\mathcal{C}^{\alpha}(\overline{\Omega}) \cap L^\infty_+(\Omega)$ with the above described properties.

Now one can easily prove the mapping  $\mathcal{C}^{\alpha}(\overline{\Omega}) \cap L^\infty_+(\Omega) \ni \sigma \mapsto u_\sigma \in W^{1,2}(\Omega)/\R$ is Lipschitz continuous if either $\tau > 0$ or $1 < p < 2$. The price one has to pay is that the dependence of the presented estimates on $f$ and $\tau$ becomes implicit.

\begin{lemma}
\label{lemma:perus2}
Let $u_{\sigma_0}, u_{\sigma_1} \in \mathcal{C}^{1,\beta}(\overline{\Omega})/\R$ be the solutions of \eqref{eq:varplaplace} corresponding to $\sigma_0, \sigma_1 \in \mathcal{B}$, respectively. If $\tau > 0$ or $1 < p \leq 2$, then
\begin{equation}
\label{eq:conv2}
\| \nabla u_{\sigma_1} - \nabla u_{\sigma_0} \|_{L^2(\Omega)} \leq C \| \sigma_1 - \sigma_0 \|_{L^\infty(\Omega)},
\end{equation}
where $C = C(\Omega, f, \mathcal{B}, p, \alpha,\tau)>0$ is independent of $\sigma_0, \sigma_1 \in\mathcal{B}$. 
\end{lemma}

\begin{proof}
We start by subtracting the variational equations \eqref{eq:varplaplace} corresponding to the conductivities $\sigma_0$ and $\sigma_1$, which yields
$$
\int_\Omega \sigma_0 \big(D\varphi_{p,\tau}(\nabla u_{\sigma_1}) - D\varphi_{p,\tau}(\nabla u_{\sigma_0})  \big) \cdot \nabla v \, {\rm d} x 
=  \int_\Omega (\sigma_0-\sigma_1) D \varphi_{p,\tau} (\nabla u_{\sigma_1}) \cdot \nabla v \, {\rm d} x.
$$
Recalling \eqref{eq:uniform} and \eqref{eq:basic_ineq} and choosing $v = u_{\sigma_1} - u_{\sigma_0}$, it follows for $p \leq 2$ or $\tau > 0$ that
\begin{align}
\label{eq:Lip}
\int_\Omega | \nabla u_{\sigma_1} &- \nabla u_{\sigma_0} |^2 \, {\rm d} x
\leq C \int_\Omega ( \tau^2 + |\nabla u_{\sigma_1}|^2 + |\nabla u_{\sigma_0} |^2)^{\frac{p-2}{2}} | \nabla u_{\sigma_1} - \nabla u_{\sigma_0} |^2 \, {\rm d} x \nonumber \\[1mm]
&\leq C \int_\Omega \sigma_0 \big(D\varphi_{p,\tau}(\nabla u_{\sigma_1}) - D\varphi_{p,\tau}(\nabla u_{\sigma_0}) \big) \cdot (\nabla u_{\sigma_1} - \nabla u_{\sigma_0}) \, {\rm d} x \nonumber \\[1mm]
&\leq C \|\sigma_1 - \sigma_0 \|_{L^\infty(\Omega)} \|  D \varphi_{p,\tau} (\nabla u_{\sigma_1})  \|_{L^2(\Omega)} \| \nabla u_{\sigma_1}  - \nabla u_{\sigma_0} \|_{L^2(\Omega)} \nonumber \\[1mm]
&\leq C \|\sigma_1 - \sigma_0 \|_{L^\infty(\Omega)} 
 \big\|  (\tau^2 + |\nabla u_{\sigma_1}|^2)^{\frac{p-1}{2}}  \big\|_{L^2(\Omega)}
\| \nabla u_{\sigma_1}  - \nabla u_{\sigma_0}\|_{L^2(\Omega)},
\end{align}
where we also used the Schwarz inequality. Together with  \eqref{eq:uniform} this proves the claim. 
\end{proof}

In fact, the assertion of Lemma~\ref{lemma:perus2} also holds for all $2 < p < \infty$ and $\tau = 0$ if $\sigma_0 \in \mathcal{B}$ is fixed and the corresponding solution $u_{\sigma_0}$ has no critical points in $\Omega$, that is, $\nabla u_{\sigma_0} \not = 0$ everywhere in $\overline{\Omega}$. Indeed, as $\nabla u_{\sigma_0}$ is H\"older continuous, it follows that actually $|\nabla u_{\sigma_0}| \geq c > 0$ and, in particular,  the first estimate in \eqref{eq:Lip} remains valid even for $\tau = 0$ and $p>2$.

The following second lemma, which is an essential building block in the next section, is a simple generalization of \cite[Lemma~3.3]{Guo16}, where only the case $\tau = 0$ is considered.

\begin{lemma}
\label{cor:perus}
Let $u_{\sigma_0}, u_{\sigma_1} \in \mathcal{C}^{1,\beta}(\overline{\Omega})/\R$ be the solutions of \eqref{eq:varplaplace} corresponding to $\sigma_0, \sigma_1 \in \mathcal{B}$, respectively. For any $\tau \geq 0$ and $0< p < \infty$,
\begin{equation}
\label{eq:conv3}
\| \nabla u_{\sigma_1} - \nabla u_{\sigma_0} \|_{L^\infty(\Omega)} \leq C \| \sigma_1 - \sigma_0 \|_{L^\infty(\Omega)}^{\epsilon}
\end{equation}
for some constants $\epsilon(\Omega, f, \mathcal{B}, p, \alpha, \tau) > 0$ and $C(\Omega, f, \mathcal{B}, p, \alpha, \tau)>0$ that are independent of $\sigma_0, \sigma_1 \in \mathcal{B}$.
\end{lemma}

\begin{proof}
The claim directly follows by applying a suitable interpolation result \cite[Lemma~A.2]{Guo16} to the combination of
$$
\| \nabla u_{\sigma_1} - \nabla u_{\sigma_0} \|_{\mathcal{C}^{\beta}(\overline{\Omega})} \leq 2 \, C_\mathcal{B}
$$
and \eqref{eq:conv}.
\end{proof}

\section{Fr\'echet derivative for \texorpdfstring{$\tau>0$}{tau>0}}
\label{sec:Frechet}
In this section we continue to assume that $\partial \Omega \in \mathcal{C}^{1,\alpha}$, $\sigma \in \mathcal{C}^{\alpha}(\overline{\Omega}) \cap L^\infty_+(\Omega)$ and the (fixed) boundary current density in \eqref{eq:plaplace} satisfies $f\in \mathcal{C}^{\alpha}(\partial \Omega) \cap L^q_\diamond(\partial \Omega)$ for some $\alpha > 0$. In addition, we only consider the case $\tau > 0$, if not explicitly stated otherwise. The aim is to prove the forward map $\mathcal{C}^{\alpha}(\overline{\Omega}) \cap L^\infty_+(\Omega) \ni \sigma \mapsto u_\sigma \in W^{1,2}(\Omega)/\R$ is Fr\'echet differentiable.

Let us consider the following linear `derivative problem': For $\eta \in L^\infty(\Omega)$, find $u_\sigma'(\eta) \in W^{1,2}(\Omega)/\R$ such that
\begin{equation}
\label{eq:deri}
\int_\Omega \sigma H_{p,\tau}(\nabla u_\sigma) \nabla u'_\sigma(\eta) \cdot \nabla v \, {\rm d} x 
= - \int_{\partial \Omega} \eta D \varphi_{p,\tau} (\nabla u_\sigma) \cdot \nabla v \, {\rm d} x
\end{equation}
for all $v \in W^{1,2}(\Omega)/\R$. As always, $u_\sigma \in \mathcal{C}^{1,\beta}(\overline{\Omega})/\R$ is the solution of \eqref{eq:varplaplace} and $H_{p, \tau}: \R^n \to \R^{n \times n}$ is the Hessian of $\varphi_{p, \tau}: \R^n \to \R$ given explicitly in \eqref{eq:Hessian}. 

\begin{lemma}
\label{lemma:LM}
The variational problem \eqref{eq:deri} has a unique solution $u_\sigma'(\eta) \in W^{1,2}(\Omega)/\R$ that satisfies
$$
\| u_\sigma'(\eta) \|_{W^{1,2}(\Omega)/\R} \leq C \| \eta \|_{L^\infty(\Omega)}
$$
where $C = C(\Omega, f, p, \alpha, \sigma, \tau)>0$ is independent of $\eta \in L^\infty(\Omega)$.
\end{lemma}

\begin{proof}
The bilinear form defined by the left-hand side of \eqref{eq:deri} is bounded, that is, 
\begin{align}
\label{eq:LMbound}
\bigg| \int_\Omega \sigma H_{p,\tau} & (\nabla u_\sigma) \nabla w \cdot \nabla v \, {\rm d} x \bigg| \nonumber \\[0.5mm]
&\leq C \|\sigma\|_{L^\infty(\Omega)} \big\|(\tau^2 + |\nabla u_\sigma|^2)^{\frac{p-2}{2}}\big\|_{L^\infty(\Omega)} \|\nabla w\|_{L^2(\Omega)} \|\nabla v\|_{L^2(\Omega) } \nonumber \\[1mm]
&\leq C \|w\|_{W^{1,2}(\Omega)/\R} \|v\|_{W^{1,2}(\Omega) /\R}
\end{align}
for all $w, v \in W^{1,2}(\Omega)/\R$ by virtue of \eqref{eq:bound} and \eqref{eq:uniform}. It is also coercive: 
\begin{align}
\label{eq:LMcouer}
\int_\Omega \sigma H_{p,\tau}(\nabla u_\sigma) \nabla v \cdot \nabla v \, {\rm d} x  & \geq {\rm ess} \inf \! \big(\sigma \, (\tau^2 + c(p) |\nabla u_\sigma |^2)^{\frac{p-2}{2}}\big)   \|\nabla v \|_{L^2(\Omega)}^2 \nonumber \\
& \geq c \|v\|_{W^{1,2}(\Omega) /\R}^2,
\end{align}
for any $v \in W^{1,2}(\Omega)/\C$ due to \eqref{eq:posdef} and \eqref{eq:uniform}. Since $|\nabla u_\sigma|$ is bounded, it is obvious that the right hand-side of \eqref{eq:deri} defines a bounded linear map on $W^{1,2}(\Omega)/\C$. To be more precise,
$$
\left|\int_{\partial \Omega} \eta D \varphi_{p,\tau} (\nabla u_\sigma) \cdot \nabla v \, {\rm d} x \right| \leq C \|\eta\|_{L^\infty(\Omega)} \|\nabla v \|_{L^2(\Omega)} 
$$
for any $v \in W^{1,2}(\Omega)/\C$. To sum up, the assertion follows from the Lax--Milgram theorem.
\end{proof}

The main theorem of this work is as follows:

\begin{theorem}
\label{thm:main}
Assume that $\tau > 0$, $\partial \Omega \in \mathcal{C}^{1,\alpha}$, $\sigma \in \mathcal{C}^{\alpha}(\overline{\Omega}) \cap L^\infty_+(\Omega)$ and $f\in \mathcal{C}^{\alpha}(\partial \Omega) \cap L^q_\diamond(\partial \Omega)$. Then, the mapping
$$
\mathcal{C}^{\alpha}(\overline{\Omega}) \cap L^\infty_+(\Omega) \ni \sigma \mapsto u_\sigma \in W^{1,2}(\Omega)/\R
$$
is Fr\'echet differentiable. The Fr\'echet derivative is given by the linear and bounded map
$$
\mathcal{C}^{\alpha}(\overline{\Omega}) 
\ni \eta \mapsto u'_\sigma(\eta) \in W^{1,2}(\Omega)/\R,
$$
where $u'_\sigma(\eta) \in W^{1,2}(\Omega)/\R$ is the unique solution of \eqref{eq:deri}.
\end{theorem}

\begin{proof}
Consider the difference of the variational equations \eqref{eq:varplaplace} for the conductivities $\sigma, \sigma + \eta \in \mathcal{B}$ and  subtract \eqref{eq:deri}. After rearranging terms, one arrives at
\begin{align}
\label{eq:frest1}
\int_\Omega  & \sigma  H_{p,\tau} (\nabla u_\sigma) \nabla \big(u_{\sigma+\eta} - u_\sigma - u'_\sigma(\eta)\big) \cdot \nabla v \, {\rm d} x  \\[1mm]
&=  \int_{\partial \Omega} \sigma \big(D\varphi_{p,\tau}(\nabla u_{\sigma}) - D\varphi_{p,\tau}(\nabla u_{\sigma+\eta}) - H_{p,\tau}(\nabla u_\sigma) (\nabla u_{\sigma} - \nabla u_{\sigma+\eta}) \big) \cdot \nabla v \, {\rm d} x   \nonumber \\[1mm]
& \quad  + \int_{\partial \Omega} \eta \big(D\varphi_{p,\tau}(\nabla u_\sigma) - D\varphi_{p,\tau}(\nabla u_{\sigma+\eta})\big) \cdot \nabla v \, {\rm d} x. \nonumber 
\end{align}
Let us estimate the two terms on the right-hand side of \eqref{eq:frest1} in turns.

By Taylor's theorem,
$$
\big| D\varphi_{p,\tau}(\x) - D\varphi_{p,\tau}(\y) - H_{p,\tau}(\x) (\x - \y)\big| \leq  C \!\! \max_{1\leq j,k,l\leq n} \max_{\z \in [\x,\y]} \left| \frac{\partial^3 \varphi_{p,\tau}}{\partial \x_j \partial \x_k \partial \x_l}(\z) \right|
| \x - \y |^2
$$
for all $\x,\y \in \R^n$. Combining this with \eqref{eq:uniform} and \eqref{eq:bound2} demonstrates that the absolute value of the first term on the right-hand side of \eqref{eq:frest1} is bounded by a constant times
\begin{align*}
\int_{\partial \Omega} \sigma |\nabla & u_{\sigma+\eta} - \nabla u_\sigma |^2 |\nabla v | \, {\rm d} x \\
&\leq C \|\nabla u_{\sigma+\eta} - \nabla u_\sigma \|_{L^\infty(\Omega)} \|\nabla u_{\sigma+\eta} - \nabla u_\sigma \|_{L^2(\Omega)} \| \nabla v \|_{L^2(\Omega)} \\[1mm]
& \leq  C \|\eta\|_{L^\infty(\Omega)}^{1+\epsilon} \| \nabla v \|_{L^2(\Omega)},
\end{align*}
where the last step follows from a combination of Lemmas~\ref{lemma:perus2} and \ref{cor:perus}.

To handle the second term on the right-hand side of \eqref{eq:frest1}, note that
$$
| D\varphi_{p,\tau}(\x) - D\varphi_{p,\tau}(\y) | \leq C | \x - \y |
$$
uniformly over any bounded set of $\R^n$ due to \eqref{eq:bound}. 
Hence,
\begin{align*}
\left|\int_{\partial \Omega} \eta \big(D\varphi_{p,\tau}(\nabla u_{\sigma+\eta}) - D\varphi_{p,\tau}(\nabla u_{\sigma})\big) \cdot \nabla v \, {\rm d} x \right| 
& \leq C \int_\Omega |\eta| |\nabla u_{\sigma+\eta} - \nabla u_{\sigma}| |\nabla v| {\rm d} x \\[1mm]
& \leq C \|\eta\|_{L^\infty(\Omega)}^2  \| \nabla v \|_{L^2(\Omega)}
\end{align*}
where we also used \eqref{eq:uniform}, the Schwarz inequality and Lemma~\ref{lemma:perus2}.

Choosing $v = u_{\sigma+\eta} - u_\sigma - u'_\sigma(\eta)$ in \eqref{eq:frest1} and combining the above estimates with \eqref{eq:LMcouer} finally yields
$$
\|u_{\sigma+\eta} - u_\sigma - u'_\sigma(\eta)\|_{W^{1,2}(\Omega)/\R}^2 
\leq C \|\eta\|_{L^\infty(\Omega)}^{1+\epsilon} \|u_{\sigma+\eta} - u_\sigma - u'_\sigma(\eta)\|_{W^{1,2}(\Omega)/\R},
$$
that is,
$$
\frac{1}{\| \eta \|_{L^\infty(\Omega)}} \, \|u_{\sigma+\eta} - u_\sigma - u'_\sigma(\eta)\|_{W^{1,2}(\Omega)/\R} \leq C \|\eta\|_{L^\infty(\Omega)}^{\epsilon}, \qquad \eta \not=0,
$$
which is a stronger version of the claim since the weaker topology of $L^\infty(\Omega)$ is used for $\eta$. 
\end{proof}

\begin{corollary}
\label{cor:main}
Assume that $\tau = 0$, the other assumptions of Theorem~\ref{thm:main} hold and, in addition, $\nabla u_\gamma$ does not vanish in $\overline{\Omega}$ for some fixed $\gamma \in \mathcal{C}^{\alpha}(\overline{\Omega}) \cap L^\infty_+(\Omega)$. Then there exists an open neighborhood $\mathcal{B} \subset \mathcal{C}^{\alpha}(\overline{\Omega}) \cap L^\infty_+(\Omega)$ of $\gamma$ such that the mapping
$$
\mathcal{B} \ni \sigma \mapsto u_\sigma \in W^{1,2}(\Omega)/\R
$$
is Fr\'echet differentiable with the corresponding derivative defined by the unique solution $u' \in W^{1,2}(\Omega) / \R$ of \eqref{eq:deri} as in Theorem~\ref{thm:main}.
\end{corollary}

\begin{proof}
First of all, if $\nabla u_\gamma \in \mathcal{C}^\beta(\overline{\Omega})$ does not vanish in $\overline{\Omega}$, then
\begin{equation}
\label{eq:lbound}
\| \nabla u_\sigma \|_{L^\infty(\Omega)} \geq c > 0
\end{equation}
for all conductivities $\sigma$ in some nonempty neighborhood $\mathcal{B}$ of $\gamma$ due to Lemma~\ref{cor:perus}. Hence, \eqref{eq:deri} has a unique solution for all $\sigma \in \mathcal{B}$ since the lower bound \eqref{eq:lbound} makes it possible to carry out the proof  of Lemma~\ref{lemma:LM} without any modification for $\tau=0$.  

The proof of Theorem~\ref{thm:main} also remains valid almost as such for $\tau =0$ whenever \eqref{eq:lbound} holds true. The only needed modifications are referring to the remark succeeding Lemma~\ref{lemma:perus2} instead of Lemma~\ref{lemma:perus2} itself at two occasions and convincing oneself that the singularity of $\varphi_{p,\tau}$ at the origin does not come into play if $\mathcal{B}$ is chosen small enough. 
\end{proof}

The following, second corollary is an easy consequence of the trace theorem for quotient spaces (cf.,~e.g.,~\cite[Lemma~2.7]{Hyvonen04}).

\begin{corollary}
\label{cor:main2}
Under the assumptions of Theorem~\ref{thm:main} (or those of Corollary~\ref{cor:main}), the mapping
$$
\mathcal{B} \ni \sigma \mapsto u_\sigma|_{\partial \Omega} \in W^{1/2,2}(\partial \Omega)/\R
$$
is Fr\'echet differentiable for $\mathcal{B} = \mathcal{C}^{\alpha}(\overline{\Omega}) \cap L^\infty_+(\Omega)$ (or for some nonempty neighborhood $\mathcal{B}$ of $\gamma$ in $\mathcal{C}^{\alpha}(\overline{\Omega}) \cap L^\infty_+(\Omega)$). The corresponding derivative is given by the linear and bounded map
$$
\mathcal{C}^{\alpha}(\overline{\Omega}) 
\ni \eta \mapsto u'_\sigma(\eta)|_{\partial \Omega} \in W^{1/2,2}(\partial \Omega)/\R,
$$
where $u'_\sigma(\eta) \in W^{1,2}(\Omega)/\R$ is the unique solution of \eqref{eq:deri}.
\end{corollary}

Take note that if one considers the differentiation of the solution to \eqref{eq:varplaplace} at $\sigma \in \mathcal{C}^\alpha(\overline{\Omega}) \cap L^\infty_+(\Omega)$ with respect to an additive perturbation $\eta\in \mathcal{C}^\alpha(\overline{\Omega})$ in some power of the conductivity $\upsilon := \sigma^r \in L^\infty(\Omega)$, $0 \not=r \in \R$, or in the log-conductivity $\kappa := \log(\sigma)$, all above conclusions remain valid if \eqref{eq:deri} is replaced by
\begin{equation}
\label{eq:deri_rho}
\int_\Omega \sigma H_{p,\tau}(\nabla u_\sigma) \nabla u'_\upsilon(\eta) \cdot \nabla v \, {\rm d} x 
= - \frac{1}{r} \int_{\partial \Omega} \eta \, \sigma^{1-r} D \varphi_{p,\tau} (\nabla u_\sigma) \cdot \nabla v \, {\rm d} x,
\end{equation}
or
\begin{equation}
\label{eq:deri_kappa}
\int_\Omega \sigma H_{p,\tau}(\nabla u_\sigma) \nabla u'_\kappa(\eta) \cdot \nabla v \, {\rm d} x 
= - \int_{\partial \Omega} \eta \, \sigma D \varphi_{p,\tau} (\nabla u_\sigma) \cdot \nabla v \, {\rm d} x
\end{equation}
respectively. This is a straightforward consequence of the chain rule for Banach spaces. For $\upsilon = \sigma^r$, the choice $r=-1$ corresponds to a parametrization with respect to the resistivity. On the other hand, $r= 1/(1-p) = -q/p$ leads arguably to a natural parametrization because then the solution to \eqref{eq:plaplace} depends linearly on a {\em homogeneous} parameter field $\upsilon$, if $\tau=0$ (cf.~\cite[Example~1]{Hyvonen18}).

We complete this section by a remark that sheds light on the difficulties one encounters if trying to prove Fr\'echet differentiability with respect to $\sigma$ when $\tau=0$ and no extra assumptions on the behavior of $\nabla u_\sigma$ are imposed.

\begin{remark}
If $\tau=0$, $p\not=2$ and $u_\sigma$ has critical points in $\overline{\Omega}$, then the coefficient matrix $H_{p,\tau}(\nabla u_\sigma)$ in \eqref{eq:deri} is either unbounded ($1< p < 2$) or without a positive definite lower bound ($2 < p < \infty$), as indicated by the estimates \eqref{eq:bound} and \eqref{eq:posdef}, respectively. There exists theory for the unique solvability of such degenerate elliptic equations~\cite{Fabes82}, but those results would typically require $|\nabla u_\sigma|^{p-2}$ to lie in a suitable Muckenhoupt class or to behave essentially like an appropriate power of the Jacobian determinant for some quasiconformal map.

Although the distribution and properties of the critical points of a solution to the (weighted) $p$-Laplace equation have been extensively studied {\em in two spatial dimensions} (see,~e.g.,~\cite{Alessandrini87,Alessandrini01,Iwaniec89,Manfredi88,Manfredi91}), the unique solvability of \eqref{eq:deri} for $\tau=0$ in an appropriate weighted Sobolev space does not seem to straightforwardly follow from,~e.g.,~the material in \cite{Fabes82} without further assumptions on the behavior of $|\nabla u_\sigma|$ close to the critical points.  
On the other hand, very little is known about the critical points of solutions to the $p$-Laplace equation in three and higher dimensions \cite{Lindqvist17}.

In addition, even if one succeeded in  proving the unique solvability of \eqref{eq:deri} for $\tau = 0$ and $p\not=2$ (under reasonable further assumptions), the proof of Theorem~\ref{thm:main} would not be valid as such but one would need to extend it in order to cover the needed weighted Sobolev spaces (see,~e.g.,~\cite[Section~2.1]{Fabes82} or \cite{Zhikov98}). As a consequence, we have decided to leave further theoretical considerations regarding the extension of Theorem~\ref{thm:main} to the case $\tau =0$ for future studies. However, most of our numerical experiments in the following section {\em do} successfully tackle the case $\tau = 0$,~i.e.,~the standard weighted $p$-Laplace equation. (Take note that finite element analysis for numerically solving degenerate elliptic boundary value problems with a coefficient in the Muckenhoupt class $A_2$ can be found in \cite{Nochetto16}.)
\end{remark}

\section{Numerical experiments}
\label{sec:numerics}

In this section, we study the numerical feasibility of the inverse boundary value problem for the $p$-Laplacian. We start by explaining how the forward problem \eqref{eq:varplaplace} and the derivative problem \eqref{eq:deri} can be solved numerically. 
Subsequently, we investigate the dependence of the forward solution $u_\sigma$  on $\sigma$ and $p$. To be more precise, we are interested in the error that results from replacing the forward map $\sigma \mapsto u_\sigma|_{\partial \Omega}$ by its linearization around $\sigma_0 \equiv 1$. Based on simulated traces $u_\sigma|_{\partial \Omega}$ for certain boundary current densities, corresponding reconstructions of $\sigma$ are finally sought via regularized least-squares minimization that employs solutions to the derivative problems \eqref{eq:deri}, \eqref{eq:deri_rho}, \eqref{eq:deri_kappa} and is motivated by the Bayesian MAP estimate. More precisely, we consider a one-step linearization approach to the inverse problem and monitor its accuracy for different values of $p$.

\subsection{Forward computations}
\label{sec:forward}

For computational simplicity, we work in two dimensions and choose $\Omega \subset \R^2$ to be the unit disk. The {\em finite element method} (FEM) with piecewise linear basis functions is employed to numerically solve the nonlinear variational problem \eqref{eq:varplaplace}. To this end, the unit disk is approximately divided into $55{,}000$ triangles that form a regular mesh with about $28{,}000$ nodes, of which $512$ are uniformly spaced along the boundary $\partial \Omega$. In the following, we implicitly assume that the inherent nonuniqueness in the considered Neumann problems \eqref{eq:varplaplace} and \eqref{eq:deri}, reflected in the use of quotient spaces in Section~\ref{sec:setting} and \ref{sec:Frechet}, is handled in some reasonable way,~e.g.,~ by fixing the value of the corresponding solutions at one node.

The conductivity $\sigma$ is discretized with $960$ non-overlapping subdomains $\Omega_i \subset \Omega$ whose closures cover the whole disk and are approximately equal in size. In each subset, $\sigma$ is constant. As a result, a discretized conductivity $\sigma$ is characterized by a vector in $\R_+^{960}$, that is,
\begin{equation}
\label{eq:pw_sigma}
\sigma = \sum_{i=1}^{960} \sigma_i \chi_i,
\end{equation}
where $\chi_i$ is the characteristic function of $\Omega_i$. In what follows, we abuse the notation by identifying a piecewise constant conductivity $\sigma \in L^\infty_+(\Omega)$ of the form \eqref{eq:pw_sigma} with the corresponding coefficient vector $\sigma \in \R_+^{960}$.

If $p=2$, the forward problem \eqref{eq:varplaplace} is linear and the solution of the associated discrete FEM problem is easily found for any mean-free boundary current density $f$ by solving a finite-dimensional linear system. For a general $1 < p < \infty$, we use the solution for the corresponding linear case as an initial guess and perform a Newton iteration to find a numerical solution to the nonlinear discrete system; see,~e.g.~\cite{Barrett93,Carstensen03,Diening08} for more information on numerically solving $p$-Laplace type equations by FEM. On the other hand, given a (discrete) forward solution $u_\sigma$ and a perturbation $\eta \in L^\infty(\Omega)$, an application of FEM to the derivative problem \eqref{eq:deri} results in a linear system for any $1 < p < \infty$. If $\tau > 0$, Lemma~\ref{lemma:LM} guarantees that this finite-dimensional system is uniquely and stably solvable, but we have not either encountered any severe numerical instabilities when solving the system for $\tau=0$ in the considered simple two-dimensional geometry.

In practice, when one tries to reconstruct $\sigma$ based on boundary measurements of $u$, one often has access to measurements for several boundary current densities $f$, i.e., for several right-hand sides in \eqref{eq:varplaplace} or in the corresponding discrete problem. Choosing the `best' densities is a task of optimal experimental design; see,~e.g.,~\cite{Huan13}. If $p=2$, the solution $u$ depends linearly on $f$ and thus linearly dependent densities do not yield any additional information on $\sigma$, at least in theory when the effect of measurement noise is not taken into account.
For general $p$, however, the situation is more complicated and the optimal setting may include linearly dependent current densities as well.
Here, we simply choose the first sixteen zero-mean trigonometric boundary currents
\begin{equation}
\label{eq:currents}
f \in \left\{ \cos(j \theta), \ \sin(j \theta) : j=1,\ldots,8 \right\},
\end{equation}
where $\theta$ is the angular coordinate on the unit circle $\partial \Omega$. The simulated boundary traces for $u_\sigma$ are $L^2(\partial \Omega)$-orthogonally projected onto the same zero-mean trigonometric basis, which is convenient since the constant component in $u_\sigma$ is not uniquely defined due to the Neumann boundary condition in \eqref{eq:varplaplace}. As a consequence, the total boundary measurement for a given triplet $(\sigma, p, \tau) \in \R^{960}_+ \times (1,\infty) \times \R_+\cup \{0\}$ can be represented as a vector $U = U(\sigma,p,\tau) \in \R^{256}$.

When computing derivatives of the solution $u$ with respect to the conductivity, we consider all $960$ elementary perturbations $\eta = \chi_i \in L^\infty(\Omega)$ of $\sigma$, supported on the subsets $\Omega_i$, $i=1, \dots, 960$, respectively.  The traces of the corresponding FEM approximations for \eqref{eq:deri} are then projected onto the aforementioned trigonometric basis. Consequently, the discretized derivatives for all sixteen  boundary current densities at $\sigma_0 \equiv 1$ can be expressed as a Jacobian matrix $J \in \R^{256 \times 960}$. Take note that Theorem~\ref{thm:main} does not actually guarantee that the unique solution to \eqref{eq:deri} with $\eta = \chi_i$ represents a derivative of the corresponding solution of \eqref{eq:varplaplace} since $\chi_i$ obviously does not belong to $\mathcal{C}^{\alpha}(\overline{\Omega})$ for any $\alpha >0$.\footnote{It is worth noting that if $p=2$, it is well known the Fr\'echet derivative with respect to the conductivity is given by \eqref{eq:deri} for any $\tau \geq 0$ and $\eta\in L^\infty(\Omega)$; cf.~e.g.,~\cite{Hyvonen18}.} However, the numerical experiments presented in the following section suggest that this is anyway the case. 

\subsection{Linearized forward map}
\label{sec:linforward}
In addition to the standard version $U: \R_+^{960} \to \R^{256}$, we also consider a parametrization of the discretized forward map with respect to a power $r \in \R$ of the conductivity, that is, 
\begin{equation}
\label{eq:power}
\R_+^{960} \ni \upsilon \mapsto U(\upsilon^{1/r}) =: U_{\rm pwr}(\upsilon,r) \in \R^{256},
\end{equation}
as well as a parametrization employing the log-conductivity $\kappa = \log(\sigma)$,
$$
\R^{960} \ni \kappa \mapsto U({\rm e}^\kappa) =: U_{\exp}(\kappa) \in \R^{256}.
$$
Here and it what follows, the dependence of the boundary measurement vector $U$ on $p$ and $\tau$ is suppressed and algebraic operations on coefficient vectors such as $\sigma$, $\upsilon$ and $\kappa$ are to be understood componentwise or through the associated piecewise constant representations of the form \eqref{eq:pw_sigma}. Moreover, we are actually only interested in two special choices for $r$ in \eqref{eq:power}, namely
$$
U_{\rm inv} := U_{\rm pwr}(\, \cdot \, , -1) \qquad {\rm and} \qquad
U_{\rm nat} := U_{\rm pwr}\big(\, \cdot \, , -q/p\big).
$$
These are the parametrizations with respect to the resistivity and the `natural power' of $\sigma$, respectively; see the discussion succeeding \eqref{eq:deri_kappa}. We denote the corresponding parameter vectors by $\rho = 1/\sigma$ and $\mu = \sigma^{-q/p}$.

Recall that the variational derivative problem associated to $U_{\rm pwr}: \R_+^{960} \to \R^{256}$ is \eqref{eq:deri_rho} and that associated to $U_{\rm exp}: \R^{960} \to \R^{256}$ is \eqref{eq:deri_kappa}. In other words, derivatives of the above introduced new parametrizations for the forward map can be estimated by solving \eqref{eq:deri_rho} or \eqref{eq:deri_kappa} in a similar manner as those of the original $U: \R_+^{960} \to \R^{256}$ can be produced by solving \eqref{eq:deri}. We denote the Jacobian matrices of $U_{\rm inv}$, $U_{\rm nat}$ and  $U_{\rm exp}$ evaluated at $\rho_0 := 1/\sigma_0 \equiv 1$, $\mu_0 := \sigma_0^{-q/p} \equiv 1$ and $\kappa_0 := \log(\sigma_0) \equiv 0$ by $J_{\rm inv}, J_{\rm nat}, J_{\rm exp} \in \R^{256 \times 960}$, respectively. 

Our aim is to statistically test the accuracy of the linearizations of $U$, $U_{\rm inv}$, $U_{\rm nat}$ and $U_{\rm exp}$ as functions of $p$ around $\sigma_0$, $\rho_0$, $\mu_0$ and $\kappa_0$, respectively. Bear in mind that $\sigma_0$, $\rho_0$, $\mu_0$ and $\kappa_0$ all define the same homogeneous coefficient in \eqref{eq:varplaplace}, so we are essentially comparing four different ways of linearizing the same forward map. To this end, we define a discrete random log-conductivity field on $\Omega$ by letting its components follow a $960$-dimensional Gaussian distribution with vanishing mean and a covariance matrix with entries
\begin{equation}
\label{eq:covmat}
\Sigma_{ij} = \varsigma^2 \exp\mathopen{}\left( -\frac{\lVert \hat{x}_i-\hat{x}_j \rVert_2^2}{2b^2} \right)\mathclose{}, \qquad i,j=1,\ldots,960.
\end{equation}
Here $\varsigma^2$ is the pointwise variance, $b>0$ specifies the correlation length in $\Omega$, $\hat{x}_i$ denotes the center of $\Omega_i$ and $\|  \cdot  \|_2$ is the Euclidean norm. We draw four log-conductivity samples, with $1{,}000$ members each, using the parameter values
\begin{itemize}
\item[(A)] $\varsigma^2 = 1/4$ and $b=1/3$, 
\item[(B)] $\varsigma^2 = 1/4$ and  $b=2/3$, 
\item[(C)] $\varsigma^2 = 1$ and  $b=1/3$,
\item[(D)] $\varsigma^2 = 1$ and  $b=2/3$
\end{itemize}
respectively, in \eqref{eq:covmat}. In what follows, $\mathbb{E}[\cdot]$ denotes the sample mean operator with respect to a generic log-conductivity sample.

We define the mean relative linearization error for the standard discrete forward map $U: \R_+^{960} \to \R^{256}$ via
\begin{equation}
\label{eq:error}
e_{{\rm std}} = \, \mathbb{E} \! \left[\frac{\big\| U(\sigma) - \big( U (\sigma_0) + J (\sigma - \sigma_0) \big) \big\|_2}{\| U(\sigma) \|_2} \right],
\end{equation}
where the conductivity is given by $\sigma = {\rm e}^{\kappa}$ as $\kappa$ runs through a log-conductivity sample (either A, B, C or D). Take note that $e_{\rm std} = e_{\rm std}(p,\tau)$ is still a function of the parameter pair $(p, \tau)$ as well as of the considered log-conductivity sample. 
The mean relative linearization errors $e_{\rm inv}$, $e_{\rm nat}$ and $e_{\rm exp}$ are defined analogously, that is, by using the appropriate forward map ($U_{\rm inv}$, $U_{\rm nat}$ or $U_{\rm exp}$), its Jacobian and the correct base point for the linearization on the right-hand side of \eqref{eq:error} with the sample variable being defined as $\rho = {\rm e}^{-\kappa}$, $\mu = {\rm e}^{- q \kappa/p}$ or simply as $\kappa$ itself. Due to Parseval's identity, an error indicator of the type \eqref{eq:error} can be interpreted as an approximation for the mean relative $L^2(\partial \Omega)$ linearization error in the boundary potentials induced by the input current densities \eqref{eq:currents}.

The relative linearization errors are illustrated in Figure~\ref{fig:forwerr}. More precisely, $e_{\rm std}(p, \tau)$, $e_{\rm inv}(p, \tau)$, $e_{\rm nat}(p, \tau)$ and $e_{\rm exp}(p, \tau)$ are plotted as functions of $p \in [3/2,3]$ for the four samples A, B, C and D as well as two values for the smoothening parameter, namely $\tau = 0$ and $\tau=0.1$. The choice of a {\em small} $\tau \geq 0$ does not seem to have any significant impact on the plots in Figure~\ref{fig:forwerr}; in fact, switching between $\tau=0$ and $\tau=0.1$ leads to nearly overlapping plots for all log-conductivity samples and linearizations. Hence, the effect of $\tau$ is not further discussed in the following.

\begin{figure}
\center{\includegraphics[scale=1]{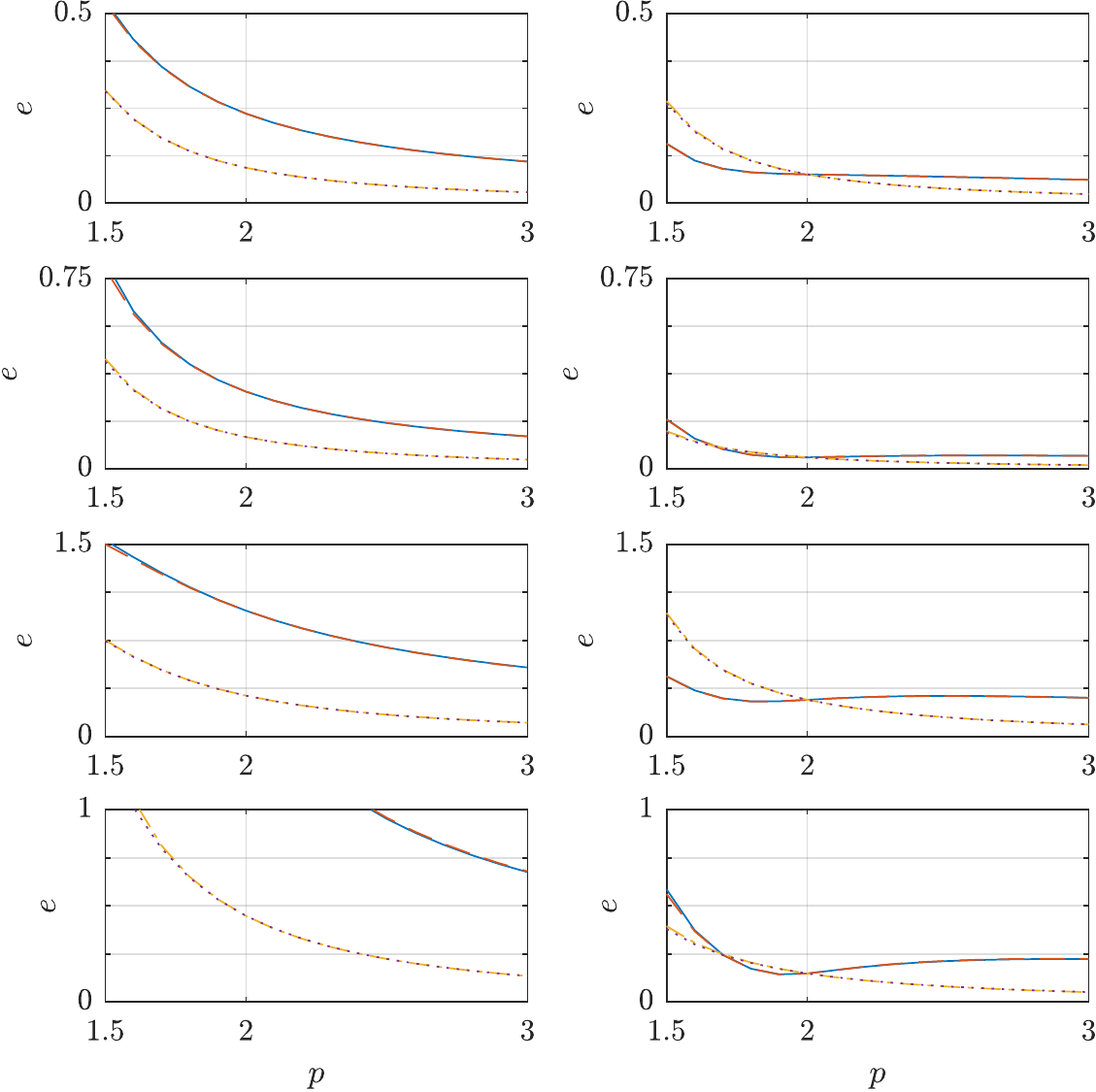}}
\caption{Mean relative linearization errors for the four parametrizations of the forward operator as functions of $p \in [3/2, 3]$. The rows correspond to the samples A, B, C and D from top to bottom. Left column: $e_{\rm std}$ and $\tau = 0$ (solid), $e_{\rm std}$ and $\tau = 0.1$ (dashed), $e_{\rm exp}$ and $\tau = 0$ (dash-dotted), $e_{\rm exp}$ and $\tau = 0.1$ (dotted). Right column: $e_{\rm inv}$ and $\tau = 0$ (solid), $e_{\rm inv}$ and $\tau = 0.1$ (dashed), $e_{\rm nat}$ and $\tau = 0$ (dash-dotted), $e_{\rm nat}$ and $\tau = 0.1$ (dotted).} 
\label{fig:forwerr}
\end{figure}

Apart from $e_{\rm inv}$ for the samples C and D, the linearization errors depicted in Figure~\ref{fig:forwerr} are (almost) monotonically decreasing in $p \in [3/2,3]$, implying that the corresponding (discrete) forward operators become more linear as $p$ increases. On the other hand, the graphs of $e_{\rm inv}$ on the bottom two rows of Figure~\ref{fig:forwerr} suggest that the `highest level of linearity' seems to occur slightly left of $p=2$ for the resistivity parametrization and the samples C and D with the larger pointwise variance $\varsigma^2 = 1$. 

The linearization with respect to the conductivity is the least accurate for all examined values of $p$ and all four log-conductivity samples; in fact, its performance is intolerably bad for the sample C and especially for D. For the samples A and C with the shorter correlation length $b=1/3$, the linearizations with respect to the resistivity, log-conductivity and the natural parameter $\mu = \sigma^{-q/p}$ are comparable in accuracy. To be more precise, the resistivity linearization is the most reliable technique for $p \in [3/2,2)$ whereas $e_{\rm nat}$ and $e_{\rm exp}$ attain smaller values than  $e_{\rm inv}$ for $p \in (2,3]$, with the `natural linearization' being the most accurate method for large $p$.
For the samples B and D with the longer correlation length $b=2/3$, the linearization with respect to $\mu$ is the most accurate for almost all $p \in [3/2,3]$ and the one with respect to the resistivity seems to also function reliably. On the other hand, the linearization with respect to the log-conductivity is reasonably accurate for the sample B with the smaller pointwise variance $\varsigma^2 = 1/4$, but with C corresponding to $\varsigma^2 = 1$ its performance deteriorates when $p$ approaches the lower limit of $3/2$.

\subsection{Linearized inverse problem}
\label{sec:inverse}

In this section, we finally tackle the inverse boundary value problem corresponding to \eqref{eq:varplaplace}. Encouraged by the numerical tests of the preceding section, we only consider the case $\tau=0$, and we concentrate mainly on the forward map $U_{\rm exp}: \R^{960} \to \R^{256}$. The reason for the latter is two-fold: First of all, the logarithmic parametrization for the conductivity ensures that the obtained reconstructions are (physically) meaningful without any further restrictions on the optimization process, that is, $\sigma = {\rm e}^\kappa$ is a positive conductivity for any $\kappa \in \R^{960}$. Secondly, of the three discrete forward operators, $U_{\rm exp}$ leads to the simplest form for the Tikhonov functional that defines the MAP estimates because the argument of $U_{\rm exp}$,~i.e.~the log-conductivity, is the variable that follows a Gaussian prior probability distribution in the considered setting. 

However, we also compare the `inverse accuracies' of all four linearization techniques introduced in the previous section for two new log-conductivity samples drawn from moderate zero-mean Gaussian distributions defined by the covariance matrix \eqref{eq:covmat} with the parameter pairs
\begin{itemize}
\item[(E)] $\varsigma^2 = 1/100$ and $b=1/3$,
\item[(F)] $\varsigma^2 = 1/100$ and $b=2/3$.
\end{itemize}
If the log-conductivity $\kappa$ follows a Gaussian distribution with such a small variance, then the log-normal distributions for $\sigma = {\rm e}^\kappa$, $\rho = {\rm e}^{-\kappa}$ and $\mu =  {\rm e}^{- q\kappa/p} \in \R^{960}$ can be approximated relatively well with Gaussian distributions having the same means and covariance matrices as the to-be-approximated log-normal ones. This allows a relatively fair comparison between the four to-be-introduced one-step reconstruction methods that are based on the four parametrizations of the forward map introduced in Section~\ref{sec:linforward}.

To simulate data for the inverse problem, we introduce noisy boundary measurements corresponding to a given log-conductivity $\kappa \in \R^{960}$ via
\begin{equation}
\label{eq:data}
V(\kappa,\omega) = U_{\rm exp}(\kappa) + \omega,
\end{equation}
where the components of $\omega \in \R^{256}$ are independent realizations of a Gaussian random variable with vanishing mean and standard deviation $\lambda>0$. We then try to reproduce (an approximation for) $\kappa$ by defining 
\begin{equation}
\label{eq:tikho}
\kappa_{\rm reco}(\kappa,\omega) = \argmin_{\widetilde{\kappa} \in \R^{960}} \left\{ \big\lVert V(\kappa,\omega) - \big( U_{\rm exp}(\kappa_0) + J_{\rm exp}(\widetilde{\kappa} - \kappa_0)\big) \big\rVert_2^2 + \lambda^2 \widetilde{\kappa}^\mathrm{T} \Sigma^{-1} \widetilde{\kappa} \right\},
\end{equation}
where $\Sigma$ is the covariance matrix defined by \eqref{eq:covmat} with appropriate choices for $\varsigma^2$ and $b$. If $U_{\rm exp}(\kappa_0) + J_{\rm exp}(\widetilde{\kappa} - \kappa_0)$ were replaced by the nonlinear forward operator $U_{\rm exp}(\widetilde{\kappa})$ itself in \eqref{eq:tikho}, a corresponding minimizer would be a MAP estimate for the discrete log-conductivity, assuming $\Sigma$ is the covariance matrix of a zero-mean Gaussian prior probability distribution for $\kappa$; see,~e.g.,~\cite{Kaipio04a}. In other words, $\kappa_{\rm reco}$ defined by \eqref{eq:tikho} is a one-step linearization approximation for a MAP estimate. Computing the minimizer in \eqref{eq:tikho} is trivial on modern computers as the involved matrices are relatively small.  It should also be mentioned that $U_{\rm exp}(\kappa_0)$ and $J_{\rm exp} = J_{\rm exp}(\kappa_0)$ are computed on a slightly different FEM mesh compared to that used for simulating the data $V$ to make sure no inverse crimes are committed.

When one-step MAP-motivated reconstructions corresponding to some other parametrization of the forward operator are considered, the employed linearization on the right-hand side of \eqref{eq:tikho} naturally deals with the investigated parametrization. Moreover, the covariance $\Sigma$ and the implicitly included zero-mean of the prior Gaussian density for $\kappa$ are replaced in the penalty term of \eqref{eq:tikho} by the covariance matrix and expectation value of the prior log-normal density for the investigated parameter. As an example, for the standard forward operator $U: \R_+^{960} \to \R^{256}$, \eqref{eq:tikho} is replaced by
\begin{align}
\label{eq:tikho2}
\sigma_{\rm reco}(\kappa,\omega) = \argmin_{\widetilde{\sigma} \in \R^{960}} \Big\{ \big\lVert V(\kappa,\omega) - \big( U(\sigma_0) &+ J(\widetilde{\sigma} - \sigma_0)\big) \big\rVert_2^2 \nonumber
\\ 
&+ \lambda^2 \big(\widetilde{\sigma} - \bar{\sigma}\big) ^\mathrm{T} \Sigma_{\sigma}^{-1} \big(\widetilde{\sigma} - \bar{\sigma}\big) \Big\}, 
\end{align}
where $\bar{\sigma} \in \R^{960}$ and $\Sigma_{\sigma} \in \R^{960 \times 960}$ are the mean and the covariance, respectively, of the log-normal prior density induced for $\sigma = {\rm e}^\kappa$ when $\kappa$ follows a priori a given zero-mean normal distribution. The estimators $\rho_{\rm reco}(\kappa,\omega)$ and $\mu_{\rm reco}(\kappa,\omega)$ corresponding to the resistivity and the natural parametrization are defined in an analogous manner. In particular, notice that the mean $\bar{\mu}$ and the covariance matrix $\Sigma_\mu$ for the natural parametrization depend on $p$ as does the relation $\mu = {\rm e}^{-q \kappa/p}$.

We test the accuracy of the above introduced simple one-step reconstruction algorithm for the inverse boundary value problem associated to \eqref{eq:varplaplace} by investigating the mean reconstruction error
\begin{equation}
\label{eq:inverror}
\iota_{\rm exp}(p) = \sqrt{\frac{\pi}{960}} \,\mathbb{E} \big[\| \kappa - \kappa_{\rm reco} \|_2\big],
\end{equation}
where the target log-conductivities run through one of the simulated samples. The corresponding error indicators for the other three parametrizations,~i.e.~$\iota_{\rm std}$, $\iota_{\rm inv}$ and $\iota_{\rm nat}$, are defined via replacing $\kappa_{\rm reco}$ in \eqref{eq:inverror} by $\log( \sigma_{\rm reco})$, $-\log( \rho_{\rm reco})$ and $(1-p) \log( \mu_{\rm reco})$, respectively. In particular, the mean reconstruction error is always measured in log-conductivity. 

All reconstructions are computed as indicated by \eqref{eq:tikho} or \eqref{eq:tikho2} with the needed covariance matrices formed using \eqref{eq:covmat} and the induced mean and covariance formulae for the log-normal distributions of $\sigma$, $\rho$ and $\mu$. The pointwise variance and correlation length needed for \eqref{eq:covmat} are the ones used when drawing the log-conductivities to which the sample mean in \eqref{eq:inverror} refers, that is, our prior information on the log-conductivity is accurate.
Because the subdomains defining the discretization for the conductivity, i.e.~$\Omega_i$, $i=1, \dots, 960$, are approximately of the same size, the reconstruction error indicators essentially measure the $L^2(\Omega)$ reconstruction error. Observe that the sample mean in \eqref{eq:inverror} also averages over the measurement noise since the additive noise vectors in \eqref{eq:data} are drawn independently for different log-conductivities in the considered samples (but they are the same for different values of $p$ to ensure fair comparison between different parameter values).

Figure \ref{fig:invcomp} compares the mean reconstruction errors $\iota_{\rm std}$, $\iota_{\rm inv}$, $\iota_{\rm nat}$ and $\iota_{\rm exp}$ for the samples E and F. Because the deviation of the considered log-conductivity samples from their mean is relatively small, the standard deviation of noise is also chosen to be rather small, i.e.~$\lambda = 10^{-3}$, to allow a reasonable signal-to-noise ratio. The largest values in the data vectors $U_{\rm exp}(\kappa)$ are of order one (and exactly one for all $p$ if $\kappa = \kappa_0 \equiv0$), so the noise level roughly corresponds to $0.1\%$ of the maximum value in $U_{\rm exp}(\kappa)$. However, it is considerably higher compared to the relative measurement $U_{\rm exp}(\kappa) - U_{\rm exp}(\kappa_0)$. Although the one-step reconstruction method for $\kappa$ takes correctly into account the prior information on the parameter field in \eqref{eq:varplaplace}, the mean reconstruction errors for the resistivity and natural parametrizations are in many cases smaller, especially for small $p$. The standard conductivity parametrization performs the worst for both samples and all $p \in [3/2,3]$. Not so surprisingly, the longer correlation length in the sample F allows on average more accurate reconstructions. All in all, the value of $p$ does not seem to have a significant effect on the average quality of the reconstructions for the considered log-conductivity samples with a small pointwise variance.

\begin{figure}
\center{\includegraphics[scale=1]{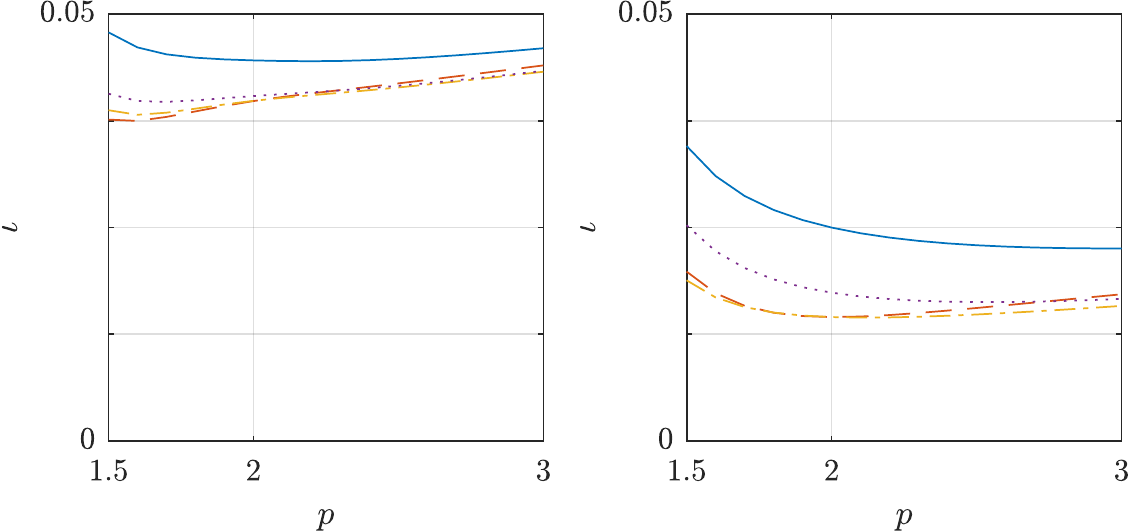}}
\caption{Mean one-step reconstruction errors $\iota_{\rm std}$ (solid), $\iota_{\rm inv}$ (dashed), $\iota_{\rm nat}$ (dash-dotted) and $\iota_{\rm exp}$ (dotted) as functions of $p \in [3/2,3]$ for $\tau=0$ and $\lambda = 10^{-3}$. Left: sample E. Right: sample~F.  The average (approximate) $L^2(\Omega)$ norm $\sqrt{\pi/960} \, \mathbb{E} [\| \kappa \|_2]$ equals $0.174$ and $0.167$ for the samples E and F, respectively.}
\label{fig:invcomp}
\end{figure}

Next we consider log-conductivity models with a larger pointwise variance. This renders the direct log-conductivity reconstruction given by \eqref{eq:tikho} the only reasonable choice, at least if the reconstruction error is still measured in log-conductivity. Figure \ref{fig:inverr} shows the mean reconstruction error $\iota_{\rm exp}$ as a function of $p \in [2/3,3]$ for the samples A and B with the pointwise variance $\varsigma^2 = 1/4$. As the considered log-conductivity samples deviate now considerably from their mean, we test a higher standard deviation for the measurement noise, $\lambda = 10^{-2}$. Recall that such a noise level approximately corresponds to $1\%$ of the maximum value in the data vectors $U_{\rm exp}(\kappa)$. Figure \ref{fig:inverr} also presents the same errors when the reconstructions of the sample log-conductivities are computed by replacing $U_{\rm exp}(\kappa_0) = U_{\rm exp}(\kappa_0, p)$ and $J_{\rm exp} = J_{\rm exp}(\kappa_0, p)$ in \eqref{eq:tikho} by $U_{\rm exp}(\kappa_0, 2)$ and $J_{\rm exp}(\kappa_0, 2)$, respectively. In other words, the measurement data are simulated using a continuum  of values for $p$, but the reconstructions are computed by choosing $p=2$ independently of its `true' value. This could be the case, for example, when one has a measurement vector available, but does not know the precise value of $p$ due to,~e.g.,~uncertainty about the correct forward model for the investigated physical phenomenon.

\begin{figure}
\center{\includegraphics[scale=1]{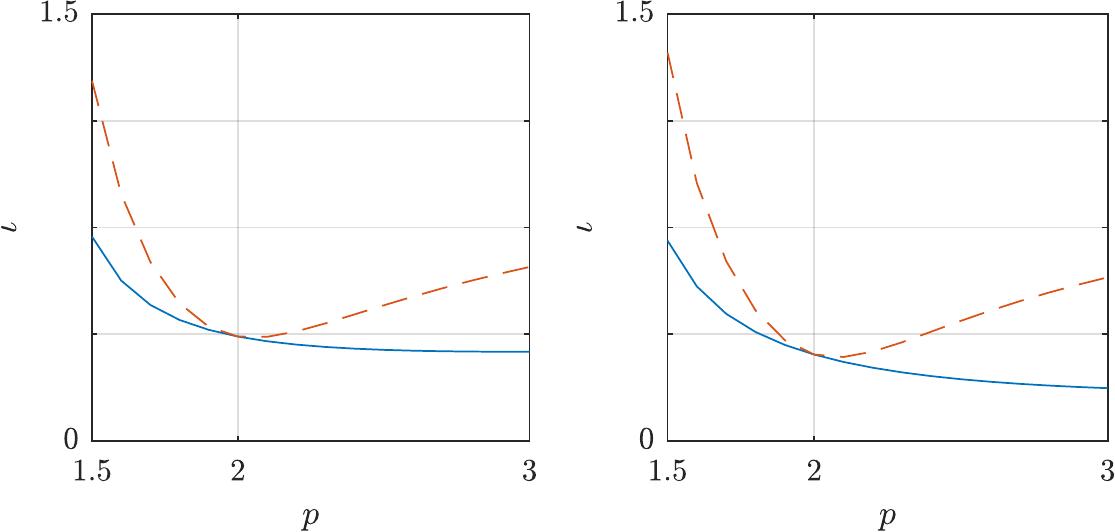}}
\caption{Mean one-step reconstruction error $\iota_{\rm exp}$ (solid) and the corresponding mean reconstruction error obtained by misusing $p=2$ in \eqref{eq:tikho} (dashed) as functions of $p \in [3/2,3]$  for $\tau=0$ and $\lambda = 10^{-2}$. Left: sample A. Right: sample B. The average (approximate) $L^2(\Omega)$ norm $\sqrt{\pi/960} \, \mathbb{E} [\| \kappa \|_2]$ equals $0.86$ and $0.84$  for the samples A and B, respectively.}
\label{fig:inverr}
\end{figure}

According to Figure~\ref{fig:inverr},  the mean reconstruction error $\iota_{\rm exp}(p)$ is monotonically decreasing in $p \in [2/3,3]$ for both log-conductivity samples, which is in line with the material in Section~\ref{sec:linforward}. Forming the reconstructions by fixing $p=2$ in \eqref{eq:tikho} produces, not so surprisingly, only slightly larger reconstruction errors close to $p=2$, but the relative performance of such a simplified approach degenerates the further one moves away from $p=2$. However, it is worth noting that fixing $p=2$ in \eqref{eq:tikho} anyway leads to smaller mean reconstruction errors for $p \in [2, 3]$ than the use of the correct $p$ does for many values in the interval $[3/2,2)$. 
According to our experience, the mean reconstruction errors could be decreased for all $p \in [2/3,3]$ by placing more weight on the penalty term in \eqref{eq:tikho}: The exact forward model is replaced by a linearization in \eqref{eq:tikho} and one can often successfully compensate for the resulting numerical error by increasing the assumed level of measurement noise (cf.~\cite{Hyvonen18}).

To conclude the numerical experiments, Figures~\ref{fig:recoA} and \ref{fig:recoB} show three example log-conductivities from the samples A and B, respectively, as well as the corresponding reconstructions produced by \eqref{eq:tikho} for $p=2/3$, $p=2$ and $p=3$ with a low noise level  $\lambda=10^{-3}$. When $p=3/2$, the reconstructions often contain log-conductivity values that are either significantly too large or too small. This overshoot phenomenon is clearly visible on the second rows of Figures~\ref{fig:recoA} and \ref{fig:recoB}, and it probably explains the bad average performance of the one-step reconstruction algorithm for small $p$ documented in Figure~\ref{fig:inverr}.  On the other hand, having $p=3$ results in the most accurate reconstructions in all shown examples, though the difference between the cases $p=2$ and $p=3$ is relatively small.

\begin{figure}
\center{\includegraphics[scale=1]{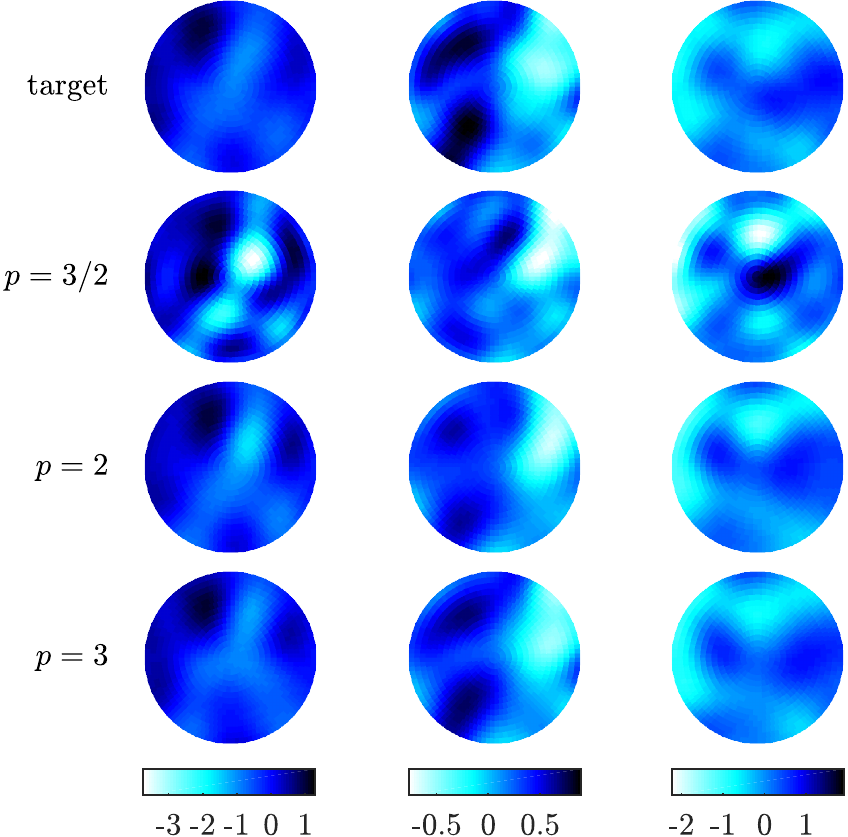}}
\caption{Example reconstructions for the sample A and $\lambda=10^{-3}$. Top row: three target log-conductivities. Second row: reconstructions with $p=3/2$. Third row: reconstructions with $p=2$. Fourth row: reconstructions with $p=3$.}
\label{fig:recoA}
\end{figure}

\begin{figure}
\center{\includegraphics[scale=1]{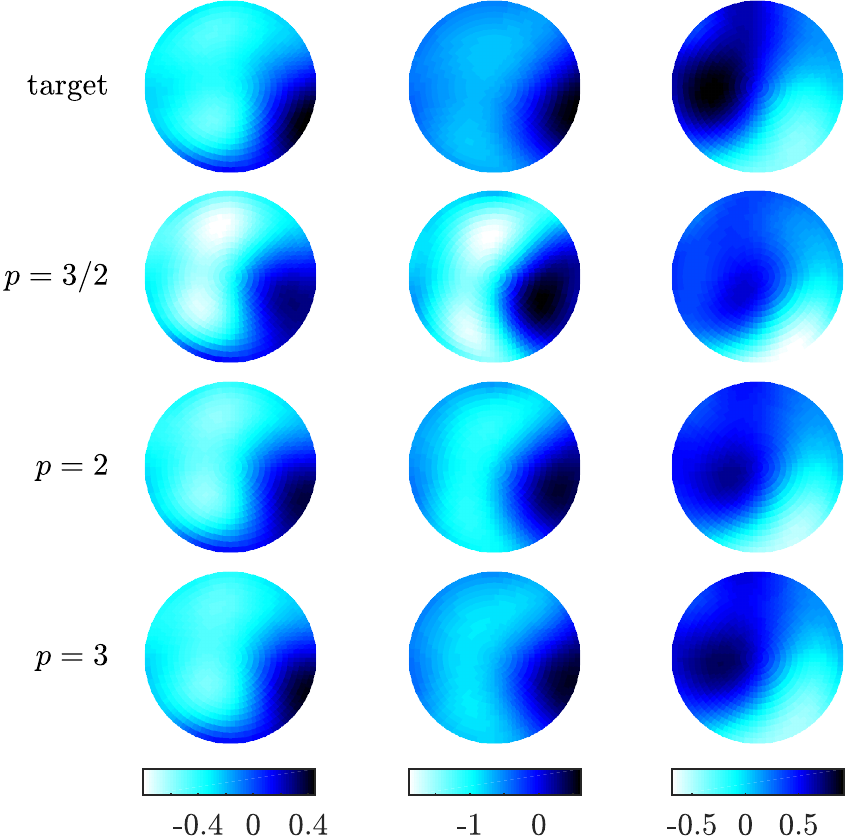}}
\caption{Example reconstructions for the sample B and $\lambda=10^{-3}$. Top row: three target log-conductivities. Second row: reconstructions with $p=3/2$. Third row: reconstructions with $p=2$. Fourth row: reconstructions with $p=3$.}
\label{fig:recoB}
\end{figure}

\section{Concluding remarks}
\label{sec:conclusion}

We have tackled an inverse boundary value problem for a family of $p$-Laplace type nonlinear elliptic partial differential equations by taking a straightforward linearization approach. In particular, the Fr\'echet differentiability of the forward operator, which maps a H\"older continuous conductivity coefficient to the solution of a Neumann problem, was established, excluding the degenerate case $\tau=0$ that corresponds to the classical (weighted) $p$-Laplacian. According to our numerical studies, the considered one-step inversion algorithm produces approximately as good, or even slightly better reconstructions for $p>2$ than it does in the case $p=2$ that corresponds to the extensively studied inverse conductivity problem. However, the accuracy of the reconstruction algorithm deteriorates for $p<2$. These conclusions hold even for $\tau=0$, but they are probably conditional to the chosen parametrization for the conductivity. Indeed, our numerical studies on the linearization error associated to the forward operator hint the conductivity parametrization has a significant effect on the nonlinearity of the considered inverse boundary value problem.

\appendix
\section{Basic properties of \texorpdfstring{$\varphi_{p,\tau}$}{phi\_\{p,tau\}}}\label{app:A}
Let $1<p<\infty$, $\tau\geq0$ and recall $\varphi_{p,\tau}$  from \eqref{eq:loc_energy} as well as its gradient \eqref{eq:varphi}. Observe that the corresponding Hessian is
\begin{equation}
\label{eq:Hessian}
H_{p,\tau}(\x) = (\tau^2 + |\x|^{2})^{\frac{p-2}{2}} I + (p-2) (\tau^2 + |\x|^2)^{\frac{p-4}{2}} \x \, \x^{\rm T},
\end{equation}
where $I \in \R^{n \times n}$ is the identity matrix.
The Hessian satisfies
\begin{align}
\label{eq:posdef}
\y^{\rm T} H_{p,\tau}(\x) \, \y &= (\tau^2 + |\x|^2)^{\frac{p-4}{2}}\big( (\tau^2 + |\x|^2) |\y|^2 + (p-2) (\x \cdot \y)^2\big) \nonumber \\[1mm]
&\geq (\tau^2 + |\x|^2)^{\frac{p-4}{2}} \big( \tau^2 + \min \{p-1,1\} |\x|^2 \big) |\y|^2 ,
\end{align}
i.e., it is uniformly positive definite with respect to $\x$ in any bounded set of $\R^n$ for a fixed $\tau > 0$. In particular, $\varphi_{p,\tau}$ is strictly convex for any $1 < p < \infty$ and $\tau \geq 0$. By choosing either $\y = \x$ or $\y\perp \x$, it is also easy to check that the spectral matrix norm satisfies
\begin{equation}
\label{eq:bound}
\| H_{p,\tau}(\x) \|_2 \leq \max\{1, p-1\} (\tau^2 + |\x|^2)^{\frac{p-2}{2}}.
\end{equation}
Furthermore, a straightforward calculation gives
\begin{equation}
\label{eq:bound2}
\max_{1 \leq j,k,l \leq n} \left| \frac{\partial^3 \varphi_{p,\tau}}{\partial \x_j \partial \x_k \partial \x_l}(\x) \right| \leq C(p) (\tau^2 + |\x|^2)^{\frac{p-3}{2}} 
\end{equation}
for all $\x \in \R^n$.

Since the graph of a convex function lies above its tangent plane, it holds that
\begin{equation}
\label{eq:convex}
\varphi_{p,\tau}(\y) \geq \varphi_{p,\tau}(\x) + D \varphi_{p,\tau}(\x) \cdot (\y-\x)
\end{equation}
for all $\x, \y \in \R$. 
We also need the inequalities
\begin{equation}
\label{eq:basic_ineq}
\big( \tau^2 + |\x|^2 + |\y|^2\big)^{\frac{p-2}{2}} | \x - \y |^2 \leq C \big( D \varphi_{p,\tau}(\x) - D \varphi_{p,\tau}(\y)\big) \cdot (\x - \y),
\end{equation}
and
\begin{equation}
\label{eq:basic_ineq2}
\big| D \varphi_{p,\tau} (\x) -  D \varphi_{p,\tau} (\y) \big| \leq C \big(\tau^2 + |\x|^2 + |\y|^2\big)^{\frac{p-2}{2}} | \x - \y|
\end{equation}
which hold for any $\x,\y \in \R^n$, $1 < p < \infty$ and $\tau \geq 0$. Indeed, \eqref{eq:basic_ineq} is a weaker version of the first inequality of \cite[(2.8)]{Diening08} for $\varphi(t) := \varphi_{p,\tau}(t \hat{x})$, with $\hat{x} \in \R^n$ being any vector of unit length. Similarly, \eqref{eq:basic_ineq2} follows with a bit of extra work from the second inequality of \cite[(2.8)]{Diening08}.

For all $\x \in \R^n$, $1 < p < \infty$ and $\tau \geq 0$, we have
\begin{align}
\label{eq:welldef}
\big | D \varphi_{p,\tau} (\x) \big|^q &\leq ( \tau^2 + |\x|^2)^{\frac{q(p-1)}{2}} 
= p \,\varphi_{p,\tau}(\x) \nonumber\\[1mm] 
&\leq 2^{p/2} \max\{\tau^2, | \x |^2 \}^{p/2} 
\leq 2^{p/2} (\tau^p + |\x|^p),
\end{align}
where $q := p/(p-1)$ is the conjugate index of $p$. For $2 \leq p < \infty$ and $\tau \geq 0$, we can deduce through the same logic that
\begin{equation}
\label{eq:isop}
\big | D \varphi_{p,\tau} (\x) \big| = ( \tau^2 + |\x|^2)^{\frac{p-2}{2}} |\x| \leq
 2^{\frac{p-2}{2}} (\tau^{p-2}|\x| + |\x|^{p-1}).
\end{equation}
On the other hand, for $1< p \leq 2$ and $\tau \geq 0$, it also holds that
\begin{align}
\label{eq:pienip}
| D \varphi_{p,\tau} (\x) \big|^q &= (\tau^2 + |\x|^2)^{\frac{q(p-2)+(q-2)}{2}} (\tau^2 + |\x|^2)^{\frac{2-q}{2}} |\x|^q \nonumber \\[1mm]
& \leq (\tau^2 + |\x|^2)^{\frac{q(p - 1) - 2}{2}} |\x|^2 = D \varphi_{p,\tau} (\x) \cdot \x,
\end{align}
where the penultimate step is a consequence of the fact that $2-q\leq0$. Finally, if $1< p \leq 2$,
\begin{align}
\label{eq:pienip2}
|\x|^p &= (\tau^2 + |\x|^2)^{\frac{p-2}{2}} (\tau^2 + |\x|^2)^{\frac{2-p}{2}} |\x|^p \nonumber \\[1mm]
&\leq 2^{\frac{2-p}{2}} (\tau^2 + |\x|^2)^{\frac{p-2}{2}} \big( \tau^{2-p} |\x|^{p} + |\x|^2 \big) \\[1mm]
&= 2^{\frac{2-p}{2}} \big( D\varphi_{p,\tau}(\x) \cdot \x + \tau^{2-p} (\tau^2 + |\x|^2)^{\frac{p-2}{2}} |\x|^p\big). \nonumber
\end{align}

\bibliographystyle{acm}
\bibliography{plapn-refs}

\end{document}